%% file: main.tex
\documentclass[a4paper, 10pt]{amsproc}
\usepackage{defs}
\usepackage{orcidlink}
\usepackage{calc}
\usepackage[dvipsnames]{xcolor}
\usepackage{tabularray}
\usepackage{empheq} % for boxing equations
\newtcolorbox{empheqboxed}{colback=white, 
    colframe=black,
    boxrule=0.25mm,
    width=\columnwidth,
    sharpish corners,
    top=-2mm, % default value 2mm
    left=2pt,
    bottom=5pt
}
\definecolor{metablue}{HTML}{0064E0}
\definecolor{metafg}{HTML}{1C2B33}
\definecolor{metabg}{HTML}{F1F4F7}
\definecolor{metabgdeep}{HTML}{D9EFFF}
\definecolor{metagreen}{HTML}{EAFFE8}
\definecolor{metagreen}{HTML}{FCFFEE}
\definecolor{metared}{HTML}{FFEAE8}
\newtheorem{theorem}{Theorem}
\newtheorem{assumption}{Assumption}
\newtheorem{remark}{Remark}

\newtheorem{problem}{Problem}
\newtheorem{proposition}{Proposition}
\newtheorem{lemma}{Lemma}
\newtheorem{corollary}{Corollary}

\renewcommand{\N}{\ensuremath{\mathbb{N}}}
\newcommand{\T}{\mathsf{T}}
\newcommand{\indicator}{\mathds{1}}

\newcommand{\identityMatrix}{\mathrm{I}}
\newcommand{\statefiltration}{\mathcal{F}}
\newcommand{\controlfiltration}{\mathcal{G}}
\newcommand{\constridx}{s}

\renewcommand{\S}{Section~}

\DeclareSymbolFont{extraup}{U}{zavm}{m}{n}
\DeclareMathSymbol{\varheart}{\mathalpha}{extraup}{86}
\DeclareMathSymbol{\vardiamond}{\mathalpha}{extraup}{87}
\DeclareMathSymbol{\varclub}{\mathalpha}{extraup}{84}
\DeclareMathSymbol{\vardspade}{\mathalpha}{extraup}{85}

\usepackage[framemethod=tikz,xcolor=true]{mdframed}

\newmdenv[backgroundcolor=metabgdeep, roundcorner=10pt, skipabove=4pt, linewidth=0pt, innertopmargin=4pt]{myframe}

\newmdenv[backgroundcolor=metabg, roundcorner=10pt, skipabove=4pt, linewidth=0pt, innertopmargin=4pt]{myOCP}

\newmdenv[backgroundcolor=metared, roundcorner=10pt, skipabove=7pt, linewidth=0pt, innertopmargin=7pt]{myalgo}
\newmdenv[%
    leftmargin=0.5cm,
    backgroundcolor=yellow!10,%
    roundcorner=5pt,%
    tikzsetting={draw=red, line width=2.0pt}%
    ]{SpecialText}%

\title[MJLS-CovSteer]{Covariance Steering of Discrete-Time Markov Jump Linear Systems with Multiplicative Noise}

\author[F. Wang]{Fangji Wang\,\orcidlink{0009-0000-8047-9901}}
\author[S. Ganguly]{Siddhartha Ganguly\,\orcidlink{0000-0003-2046-2061}}
\author[P. Tsiotras]{Panagiotis Tsiotras\,\orcidlink{0000-0001-7563-4129}}

\thanks{%
	F. Wang, S. Ganguly and P. Tsiotras are with \faGroup\ Daniel Guggenheim School of Aerospace, \faUniversity\ Georgia Institute of Technology, \faMapMarker\  Atlanta, USA.}

\thanks{%
	Contact Information: (FW) \faEnvelope\ \texttt{
fwang406@gatech.edu}, (SG) \faHome\ \url{https://sites.google.com/view/siddhartha-ganguly}, \faEnvelope\ \texttt{sganguly41@gatech.edu}, (PT) \faHome\ \url{https://dcsl.gatech.edu/tsiotras.html}, \faEnvelope\ \texttt{tsiotras@gatech.edu}.
}

\begin{document}

\maketitle

\begin{abstract}
We study a finite-horizon covariance steering problem for discrete-time Markov jump linear systems (MJLS) with both state- and control-dependent multiplicative noise. The objective is to minimize a quadratic running cost while steering the system from given mode-conditioned initial means and covariances to a prescribed terminal mean and covariance. We first show that, without loss of generality, feasible controls may be represented by mode-dependent linear feedback together with feedforward and independent random components, and we highlight that, in contrast to the case without multiplicative noise, a purely affine state-feedback law does not in general suffice. To this end, we introduce a lifted-state formulation that embeds the mean and covariance information into a unified second-moment description, and we prove that the resulting lifted problem is equivalent to the original covariance steering problem formulation. This leads to a lossless relaxation in moment variables and an SDP reformulation for the unconstrained case. We further study chance-constrained covariance steering with ball and half-space constraints on the state and control, derive tractable sufficient convex surrogates, and establish an iterative reference-update scheme to reduce conservatism. Numerical experiments on a finance application illustrate our results.
\end{abstract}

\input{sections/1-Introduction}

\input{sections/2-Problem_Setup}

\input{sections/3-Main_Result}

\input{sections/4-Numerics}
\input{sections/Conclusion}

 \appendix
\input{sections/appendix}

\bibliographystyle{amsalpha}
\bibliography{refs}

\end{document}

%% file: sections/1-Introduction.tex
\section{Introduction}\label{sec:introduction}

Markov jump linear systems (henceforth referred to as MJLS) provide a natural modeling framework for controlled dynamical processes that evolve according to a finite collection of linear modes, which may undergo abrupt, random regime changes governed by a {M}arkov chain; we refer the reader to \cite{ref:Discrete:MJLS:book} for a comprehensive introduction to MJLS and its controlled and noisy variants. 
Because many engineering and economic systems operate under changing environments, failures, mode switches, or regime-dependent parameters, MJLSs have become a standard tool and dynamical setting in stochastic control and systems theory \cite{ref:indefinite:Quad:Opt:Con:MJLS:Aut07}. 
They have been used to model, among other things:
\begin{itemize}[label= \(\circ\),leftmargin=*]
    \item aerospace \cite{ref:MJLS:applications:aerospace} and robotics \cite{ref:MJLS:applications:robotics},
    \item power systems \cite{ref:MJLS:applications:power:systems} and communication networks \cite{ref:MJLS:applications:comm:networks},
    \item fault-tolerant control architectures \cite{ref:MJLS:applications:fault:detect},
    \item macroeconomic \cite{ref:MJLS:applications:economics} and financial systems \cite{ref:MJLS:finance:IJRNC:1},
\end{itemize}
as well as a variety of other engineering processes in which abrupt structural change cannot be neglected, and a control (perhaps under constraints) synthesis is necessary.

In many applications, however, random switching is only one source of uncertainty. 
A second important source of uncertainty is stochastic noise, whose effect scales with the current state or the applied control, which is more appropriately represented by \emph{multiplicative noise} rather than purely additive noise. 
The multiplicative noise regime of MJLSs is relevant, for example, when uncertainty intensifies with operating level, exposure, or actuation magnitude, as in regime-dependent financial systems, uncertain actuators, and stochastic systems with parameter fluctuations \cite{ref:LQ:OptCon:MJLS:Mult:noise:SCL24,ref:Constrained:con:MJLS:Mult:noise:SCL24,ref:DRO:finance}. 
Consequently, MJLSs with multiplicative noise have attracted sustained attention in the literature, where they have been studied from the viewpoints of stochastic LQ control \cite{ref:LQ:OptCon:MJLS:Mult:noise:SCL24}, mean-variance control \cite{ref:MJLS:finance:IJC:mean-variance:2}, \(\mathcal{H}_2\)--\(\mathcal{H}_{\infty}\) analysis and control \cite{ref:Fas:Switching:H2:MJLS:SICON}, constrained model predictive control \cite{ref:MJLS:MPC:SCL:14}, partial observation, and more recently reinforcement learning and asynchronous/hidden-mode designs \cite{ref:MJLS:Filter:II}, \cite{ref:LQ:OptCon:MJLS:Mult:noise:SCL24}.

% I will move this to some appropriate place later

% The preceding works are important literature in MJLS and control since they establish tractable algorithmic architectures and clarify how Markov jumps and multiplicative uncertainties interact through Riccati equations, second moments, LMIs, and receding-horizon optimization. At the same time, they typically focus on regulation, quadratic performance, or mean-variance criteria, whereas the present paper studies a quadratic-cost control problem with prescribed terminal mean and covariance constraints, and later incorporates chance constraints on state and control.

Driven by the necessity to quantify and regulate uncertainty in controlled processes, recent research has focused increasingly on the distributional evolution of stochastic system trajectories. A foundational subset of this field is \emph{covariance control}, which originated in the mid-1980s. A body of early works in this direction primarily investigated how state feedback could be employed to assign specific state covariances over infinite time horizons; see \cite{ref:CovSteer:I,ref:CovSteer:II,ref:CovSteer:III,ref:CovSteer:IV} for more details. 
Lately, there has been a substantial shift toward \emph{finite-horizon} covariance steering for linear stochastic systems. 
These approaches typically aim to hit specific terminal covariance targets \cite{liu2024optimal}, often while operating under chance constraints \cite{ref:CovSteer:VII:Okamoto}, \cite{ref:KO:MG:PT:Opt:CovCont:LCSS} to keep states and inputs within safe limits. 
In discrete-time regimes, much of the existing work has proceeded through optimization in an augmented state space, together with a convexification or relaxation of the resulting constraints \cite{ref:KJ:PT:Com:Efi:CDC}. Other directions include optimal transport-driven Wasserstein \cite{ref:CovSteer:VIII:Wass:Term,ref:parmar2026robust} and Gromov-Wasserstein terminal penalties \cite{ref:CovSteer:X:DD:CDC:NH:SG:KK,ref:NH:SG:KM:KK:Formation:shape:GW}, Hilbert-Schmidt penalty \cite{ref:CovSteer:IX:AH:TS:Fixed}, receding-horizon covariance control, optimization-based steering of the full state distribution, and a substantial finite-horizon continuous-time literature \cite{ref:YC:TTG:MV:Part-I,ef:YC:TTG:MV:Part-II,ef:YC:TTG:MV:Part-III,liu2024optimal,ref:Liu:PT:Mult:Noise:CT:CovSteer}. Covariance steering architectures have also been adapted in data-driven and distributionally robust settings as well; see \cite{ref:DDCovSteer:JP:PT,ref:DistRob:CovSteer}.

\subsection{Natural questions and motivation}\label{subsec:intro:question:motiv}

In the absence of multiplicative noise, a recent work \cite{shrivastava2025chance} restricted the control policy to affine state-feedback laws and solved the resulting nonconvex problem in a conservative manner. 
However, the sufficiency of this parametrization has not been established, and, to the best of our knowledge, no approach solves the problem to optimality, even without chance constraints.
Motivated by these limitations, and by the structural differences introduced by state- and control-dependent stochastic perturbations, in this paper, we consider the more general setting of MJLS with multiplicative noise. Given the preceding premise, several natural questions arise:

\begin{myOCP}
\vspace{1mm}
\begin{enumerate}[label={\textup{(\(\mathsf{Q}\)-\alph*)}}, leftmargin=*, widest=b, align=left]
    \item \label{ques:1} Does it make sense to formulate covariance steering problems for MJLSs with multiplicative noise? 
    \item \label{ques:2} Are such formulations relevant for practical applications?
    \item \label{ques:3} If so, how should covariance steering be posed for systems with regime jumps and state- and control-dependent stochastic perturbations?
    \item \label{ques:4} Is the family of affine state-feedback controllers sufficient to characterize feasible and/or optimal solutions in this setting?
\end{enumerate}
\end{myOCP}

We answer~\ref{ques:1} and \ref{ques:2} positively: covariance steering for MJLSs with multiplicative noise is both practically relevant and mathematically tractable. Question~\ref{ques:3} is answered by developing a concrete problem formulation and a solution framework based on a lifted second-moment representation along with an SDP-based reformulation. 
In contrast, Question~\ref{ques:4} is answered, in general, \emph{negatively}: unlike covariance steering problems without multiplicative noise, the family of affine state-feedback controllers is not sufficient, in general, to characterize feasible or optimal solutions. 
This structural distinction is one of the main reasons a more general control parameterization and lifted formulation are required. 
This brings us to the main contributions of this article.

% The present work addresses these questions in the \embf{affirmative}. The contributions of our article answer the questions \ref{ques:1}--\ref{ques:4} in three ways: \textbf{(a)} they show that covariance steering for MJLSs with multiplicative noise is a practical and solvable problem class, since the combination of regime jumps and state/control-dependent stochastic perturbations leads naturally to terminal moment steering problems\fangji{What is ``a meaningful problem class''}\sid{changed it to practical and solvable}; \textbf{(b)} they show that the setting is practically relevant by covering formulations motivated by applications such as dynamic hedging and by incorporating chance constraints on the state and control; \textbf{(c)} they provide a concrete formulation and solution framework for this class of problems through a lifted second-moment representation and SDP-based reformulations. The chief contributions of this article are as follows.
%%%%%%%%%%%% our contributions %%%%%%%%%%%%

\subsection{Contributions}\label{subsec:intro:contrib}
\begin{enumerate}[leftmargin=*, widest=b, align=left]

\item \textbf{Setting:} For a broad class of discrete-time Markov jump linear systems with \emph{both} state- and control-dependent multiplicative noise, we formulate a finite-horizon covariance steering problem, with mode-conditioned initial means and covariances, prescribed terminal mean and covariance, and a quadratic running cost. This is a novel class of problems not yet considered in the literature. Our formulation also admits a terminal covariance upper-bound variant, which is relevant in applications such as dynamic hedging and risk management.

\item \textbf{Control parametrization:} We show that, for the considered class of covariance steering problems, it is without loss of generality to consider a control law
% \fangji{I think it may be better to say ``consider a control law ...''. I'm afraid the "admissible" here may mislead readers that: for the case without multiplicative noise, any admissible control can be realized by an affine state-feedback law. This is not true, since for a general admissible control, a random term is also needed.}
consisting of mode-dependent linear feedback, a feedforward term, and an additional independent random component. 
This result is crucial, since unlike the case without multiplicative noise, a purely affine state-feedback law does not, in general, suffice to realize the required terminal moments.
% \sid{Fangji: can you add a reference?}. \fangji{\cite{liu2024optimal} is for linear system. I think~\cite{shrivastava2025chance} is the only paper for MJLS covariance steering but it doesn't prove this.} 
Thus, the multiplicative-noise setting is structurally different from the standard covariance-steering formulation and requires a more general representation of the feedback law.

\item \textbf{A new lifted-state reformulation:} To address the previous difficulty, we introduce a lifted-state formulation and 
% in which the mode-dependent means and covariances are embedded into a single \fangji{``mode-dependent''?} second-moment variable. We then
show that the resulting lifted moment-steering problem is \emph{exactly} equivalent to the original covariance-steering problem. This lifted formulation preserves the coupling between the mean and covariance dynamics and provides the main structural device used in the remainder of the paper. 
Based on the proposed lifted formulation, we derive a lossless relaxation in terms of moment variables and obtain an \emph{exact semidefinite programming reformulation (SDP)} for the unconstrained problem. The same framework also accommodates the terminal covariance inequality case.

\item \textbf{Chance-constrained extension and algorithm:} 
We tackle chance-constrained covariance steering for the same MJLS setting. For ball and half-space chance constraints on the state and control, we derive tractable sufficient convex surrogates and show that the resulting optimization problem can be written as an SDP whose solution is feasible for the original chance-constrained problem. To reduce the conservatism of these surrogates, we also propose an iterative reference-update scheme with tolerance relaxation.

\item \textbf{An application in finance:} Finally, we illustrate the framework on a dynamic hedging problem with regime switching, volatility-scaled state uncertainty, and control-dependent multiplicative noise, showing the numerical fidelity of our established framework.
\end{enumerate}

% \sid{Fangji: please read the contributions carefully and check.}
% \fangji{Looks pretty good in general}

\subsection{Perspectives and relevance}\label{subsec:intro:prespective}

The class of problems studied here may be viewed as a finite-horizon stochastic control problem with terminal moment specifications for systems subject to both regime switching and multiplicative uncertainty. This viewpoint is natural when one wishes to combine quadratic control performance with prescribed terminal mean and covariance requirements. A particularly relevant application, already noted in the stated contributions, is dynamic hedging, where the state may represent a tracking error, the Markov chain captures changes in market regimes, and the multiplicative terms model volatility or friction effects that scale with the portfolio state or executed trades. In broad strokes, since MJLS models (with or without multiplicative noise) arise in areas such as communication networks, robotics, fault-tolerant systems, flight systems, and finance, our framework is useful in settings where regime-dependent dynamics and state/control-dependent stochastic effects must be handled together with terminal moment or chance constraints. 
% For this reason, we believe that the results of this paper are of practical relevance. A natural next step is to apply our algorithm to such application domains and to further refine the algorithmic architecture to suit their specific requirements.

% \sid{I think we should discuss Srivastava-Oguri somewhere in the intro (max 2 lines). Maybe in the contributions? }

% \fangji{I'm thinking of putting it at the top of I.A?
% Maybe we can put it in this way (?):
% In the absence of multiplicative noise, recent work such as~\cite{shrivastava2025chance} restricts the control policy to affine state-feedback laws and solves the resulting nonconvex problem in a conservative manner.
% However, the sufficiency of this parametrization has not been established, and, to the best of our knowledge, there is no approach that solves even the problem without chance constraints to optimality under this restriction.
% Motivated by these limitations, and by the structural differences introduced by state- and control-dependent stochastic perturbations, we consider in this paper the more general setting of MJLS with multiplicative noise. This raises several fundamental questions:
% (Q-a)-(Q-c)
% (Q-d) Is the family of affine state-feedback controllers sufficient to characterize admissible and optimal solutions in this setting?
% }

%% file: sections/2-Problem_Setup.tex
\section{Problem formulation}\label{sec:prob:form}
Before formulating the primary problem, we introduce the notation that will be used throughout this article. 

% ===========================
% ====== notations ==========
% ===========================

\subsection{Notation}\label{subsec:notation}
We employ standard notation in this article. 
Let \(\N \Let \aset{1,2,\ldots}\) 
be the set of positive integers. 
Given \(d \in \N\), we will denote the standard Euclidean vector space by \(\R[d]\), which is assumed to be equipped with the standard norm \(\R[d] \ni x\mapsto \norm{x}\). Given \(T \in \N\) and a sequence $\{r_k\}_{k=0}^T$, we simply employ $r$ to denote the entire sequence, i.e., $r \Let \aset[]{r_k \suchthat k=0,\ldots,T}$.
Given a finite index set $\indexset$, and a sequence
$\{r_k(j)\}_{k=0,\ldots,T,\; j\in \indexset}$, we use
\(r_k(\indexset) \Let \{r_k(j)\}_{j\in\indexset}\) to denote the collection of quantities at time \(k\), and \(r(\indexset) \Let \bigcup_{k=0}^{T} r_k(\indexset)\) to denote the collection of these quantities over the horizon \(k=0,\ldots,T\). 
Given \(p \in \N\), we denote by \(\mathrm{I}_{p}\)  the \(p\times p\) identity matrix. 
For \(m,n \in\N\), for a given matrix $A\in \R[m\times n]$ and $p,q\in\N$ such that $p\le m$ and $q\le n$, we employ $(A)_{1:p,1:q}$ to represent the top left $p\times q$ block of $A$. 
For $q \in \N$, we denote by $\mathbb{S}^q$ and $\mathbb{S}_+^q$ the sets of $q \times q$ real symmetric and real symmetric positive semidefinite matrices, respectively.

% =============================
% ====== the problem ==========
% =============================

\subsection{The central problem and initial results}
% \sid{I think this looks better now.}

Fix the horizon $T \in \N$ and
consider the filtered probability space $\bigl(\Omega,\statefiltration,\{\statefiltration_k\}_{k=0}^T,\mathbb{P}\bigr)$, where $\statefiltration_k \subseteq \statefiltration$ for each $k=0,\ldots,T$. Fix $N_{\indexset}, n_x, n_u, n_w, m_x, m_u \in \mathbb{N}$ and define the finite index set $\indexset \coloneqq \{1,\ldots,N_{\indexset}\}$. For each $k=0,\ldots,T-1$, consider the discrete-time Markov jump linear system with additive and multiplicative noise
\begin{align}
    &x_{k+1} = A_k(q_k)x_k + B_k(q_k)u_k 
    + \sum_{r=1}^{m_x} A_k^{(r)}(q_k) x_k \, \xi_k^{(r)} \nonumber\\
    &\quad + \sum_{s=1}^{m_u} B_k^{(s)}(q_k) u_k \, \eta_k^{(s)} 
    + D_k(q_k) w_k.
    \label{eq:plain_mjls_system}
\end{align}
along with the data:
\begin{enumerate}[label={\textup{(\eqref{eq:plain_mjls_system}-\alph*)}}, leftmargin=*, widest=b, align=left]

\item \label{prob:data:1}
For $k=0,\ldots,T$, the state $x_k \in \mathbb{R}^{n_x}$ is $\statefiltration_k$-adapted~\cite{ref:Shreve:Karatzas}, the control input is $u_k \in \mathbb{R}^{n_u}$, and the mode $q_k \in \indexset$ evolves as a Markov chain with transition probabilities
$\mathbb{P}(q_{k+1}=j \mid q_k=i) \eqqcolon p_k^{ij}$.
For $k=1,\ldots,T$, $x_k$ and $q_k$ are independent conditional on $q_{k-1}$.

\item \label{prob:data:2}
For $k=0,\ldots,T-1$ and $j\in\indexset$, the disturbance $w_k \in \mathbb{R}^{n_w}$ is independent of $\statefiltration_k$, conditional on $q_k=j$, and satisfies
\[
\mathbb{E}[w_k \mid q_k=j] = 0, 
\quad 
\mathbb{E}[w_k w_k^\top \mid q_k=j] = \identityMatrix_{n_w}.
\]

\item \label{prob:data:3} For $k=0,\ldots,T-1$ and $j\in\indexset$, the multiplicative noises $\{\xi_k^{(r)}\}_{r=1}^{m_x}$ and $\{\eta_k^{(s)}\}_{s=1}^{m_u}$ are mutually independent, conditional on $q_k=j$, and satisfy
\[
\mathbb{E}\bigl[\xi_k^{(r)} \mid q_k=j\bigr] = 0,
\quad
\mathbb{E}\bigl[(\xi_k^{(r)})^2 \mid q_k=j\bigr] = 1,
\]
\[
\mathbb{E}\bigl[\eta_k^{(s)} \mid q_k=j\bigr] = 0, \quad
\mathbb{E}\bigl[(\eta_k^{(s)})^2 \mid q_k=j\bigr] = 1,
\]
for all $r=1,\ldots,m_x$ and $s=1,\ldots,m_u$. Moreover, the $\sigma$-algebra generated by the union $\{\xi_k^{(r)}\}_{r=1}^{m_x}\cup\{\eta_k^{(s)}\}_{s=1}^{m_u}$
is independent of $\statefiltration_k$ and $w_k$, conditional on $q_k=j$. 
These random variables model state- and control-dependent uncertainties, respectively.
\end{enumerate}

\begin{remark}
The class of systems in \eqref{eq:plain_mjls_system} is substantially general and strictly contains several covariance steering models studied in the literature. Indeed, if \(\indexset\) is a singleton, there is no regime switching, and \eqref{eq:plain_mjls_system} reduces to a standard discrete-time linear system with additive and multiplicative noise \cite{liu2024optimal,ref:Liu:PT:Mult:Noise:CT:CovSteer}. 
If the multiplicative-noise coefficient matrices satisfy \(A_k^{(r)}(j)=0\), and \(B_k^{(s)}(j)=0\) for all \(k,j,r\), and \(s\), then \eqref{eq:plain_mjls_system} reduces to a discrete-time MJLS without multiplicative noise \cite{shrivastava2025chance}. 
The relevance of this generalization has already been discussed in \S\ref{sec:introduction}.
\end{remark}

For all $k\in\{0,\ldots,T\}$ and $j\in\indexset$, we define the \emph{prior probability} of $j$ at $k$ as $\rho_k(j)\Let\mathbb{P}(q_k=j)$.  
The \emph{conditional mean} of $x_k$ is defined as $\mu_k(j)\Let \mathbb{E}[x_k\mid q_k=j]$, and the \emph{conditional covariance} is defined as $\Sigma_k(j) \Let \mathbb{E}\big[(x_k-\mu_k(j))(x_k-\mu_k(j))^\T\mid q_k=j\big]$. To ensure the preceding quantities are well-defined, we need following assumption.

\begin{assumption}\label{assumption:rho_positive}
For all $k=0,\ldots,T$ and $j\in\indexset$, we assume that $\rho_k(j)>0$.
\end{assumption}

\begin{remark}[On Assumption \ref{assumption:rho_positive}]
We employ Assumption \ref{assumption:rho_positive} to avoid degeneracies in the mode-conditioned quantities \(\mu_k(j)\), \(\Sigma_k(j)\), and in the backward transition probabilities \(s_k^{ij}\), which are employed in the sequel for our analysis. 
Such positivity assumptions are standard when one works with active/reachable mode sets; see, e.g., \cite{ref:Assump:NAR:CEF:SICON}, where the relevant set of Markov states is defined from those with positive initial distribution together with those reachable from them. 
Also see \cite{ref:Assump:lee2006uniform,ref:Assump:wang2013distributed} for similar mode ergodicity assumptions. 
Assumption~\ref{assumption:rho_positive}
also makes sense in practice: for a regime-switching financial model with modes representing normal, volatile, and distressed market conditions, it is natural to assign positive initial probability to each regime and to allow transitions among them with positive probability, so that every regime remains possible at each time step \cite{ref:Assump:Finance}. 
\end{remark}

The unconditional \emph{mean} and \emph{covariance} of $x_k$ are defined as
\begin{align}\label{eq:mean_covariance_relationship}
    &\mu_k:=\mathbb{E}[x_k]=\sum_{j\in\indexset} \rho_k(j)\mu_k(j),
    \\
    &\Sigma_k:=\mathbb{E}\bigl[(x_k-\mu_k)(x_k-\mu_k)^\T\bigr]=\sum_{j\in\indexset} \rho_k(j)\big(\Sigma_k(j)
     +(\mu_k(j)-\mu_k)(\mu_k(j)-\mu_k)^\T\big).
\end{align}
We assume that, at time $k$, the controller has access to the information in $\statefiltration_k$ and to the current mode $q_k$. To this end, define the filtration
\[
\controlfiltration_k \Let \statefiltration_k \vee \sigma(q_k), \quad \text{for }k=0,\ldots,T,
\]
where $\cdot \vee\cdot$ denotes the smallest $\sigma$-algebra that contains the union of the two $\sigma$-algebras, and \(\sigma(\cdot)\) denotes the $\sigma$-algebra generated by random variables or random vectors. A control sequence \(u \Let \aset[]{u_k}_{k=0}^{T-1}\) is said to be \emph{admissible} if each \(u_k\) is \(\controlfiltration_k\)-measurable and square-integrable, i.e., \(\mathbb{E}\bigl[\norm{u_k}^2\bigr]< +\infty\). 
Thus, the set of admissible controls is given by
\begin{align}
\mathcal{U} \Let \left\{u \Let \aset[]{u_k}_{k=0}^{T-1}\;\middle\vert\;  
\begin{array}{@{}l@{}}
u_k \text{ is } \controlfiltration_k\text{-measurable }\mathbb{R}^{n_u}\text{-valued random vector,}\\ 
\text{and }\mathbb{E}[\|u_k\|^2]<+\infty \text{ for each } k.
\end{array}
\right\}.\nn
\end{align}
With the preceding ingredients, we are now in a position to state the central problem studied in this article.

%%%%%%%%% original problem %%%%%%%%%%
\begin{problem}[Covariance steering for MJLS]\label{problem:original}
Fix a horizon $T \in \N$. For all $j\in\indexset$ and $k=0,\ldots,T-1$, let $Q_k(j)\succeq 0$ and $R_k(j)\succ 0$ be state and control weighting matrices, respectively; $\mu_\mathrm{in}(j)$, $\Sigma_\mathrm{in}(j)$, $\mu_\mathrm{out}$, $\Sigma_\mathrm{out}$ are given boundary conditions. Given the problem data \ref{prob:data:1}--\ref{prob:data:3} and the preceding ingredients, over the sequence of admissible control actions \(u \Let \aset[]{u_k}_{k=0}^{T-1} \in \mathcal{U}\), we consider the generalized covariance steering problem:
\begin{empheqboxed}
\begin{align}
    \min_{u \in \admcon} \quad & J \Let \sum_{k=0}^{T-1}\mathbb{E}\big[x_k^\T Q_k(q_k)x_k+u_k^\T R_k(q_k)u_k\big],
    \label{eq:cost}
    \\
    \sbjto\quad &\textup{dynamics }\ref{eq:plain_mjls_system},\label{eq:dynamics_equation_in_original_problem}
    \\
    & \mu_0(j)=\mu_\mathrm{in}(j),\; \Sigma_0(j)=\Sigma_\mathrm{in}(j) \textup{ for }j\in \indexset, \label{eq:initial_condition}
    \\
    &\mu_T=\mu_\mathrm{out},\; \Sigma_T=\Sigma_\mathrm{out}.
    \label{eq:terminal_condition}
\end{align}
\end{empheqboxed}
\end{problem}

\begin{remark}
    While Problem~\ref{problem:original} specifies an exact terminal covariance equality constraint $\Sigma_T = \Sigma_\mathrm{out}$, the proposed framework can be seamlessly adapted to handle a terminal covariance upper bound $\Sigma_T \preceq \Sigma_\mathrm{out}$. This inequality formulation is particularly relevant in dynamic hedging and risk management, where the goal is to ensure that the terminal tracking error variance remains below a prescribed risk tolerance. 
% \red{
    In addition, the proposed framework naturally accommodates mode-conditioned terminal specifications. In particular, one may impose terminal constraints of the form $\mu_T(j) = \mu_\mathrm{out}(j)$ and $\Sigma_T(j) = \Sigma_\mathrm{out}(j)$ for each mode $j \in \mathcal{S}$, or corresponding upper-bound variants $\Sigma_T(j) \preceq \Sigma_\mathrm{out}(j)$. All subsequent theoretical guarantees and linear optimal control structures remain valid under these generalizations, requiring only minor modifications to the terminal conditions in the final semidefinite programming (SDP) formulations. To avoid notational complexity and to keep the main message clear, we restrict our attention to the simpler case. 
\end{remark}

We say that the control $u$ is \emph{feasible} for Problem~\ref{problem:original} if it is admissible and the constraints~\ref{eq:dynamics_equation_in_original_problem}--\ref{eq:terminal_condition} are satisfied.
Moreover, we say that the control 
$u^\star \Let \aset[]{u_k^\star}_{k=0}^{T-1}$ is \emph{optimal} for Problem~\ref{problem:original} if it is feasible and the cost under any other feasible control is no less than the cost under $u^\star$. We enforce the following assumption for well-posedness. 
\begin{assumption}  \label{assumption:feasible_and_sigma_positive}
There exists an optimal control for Problem~\ref{problem:original}.
\end{assumption}

Next, we state the first result for Problem~\ref{problem:original}.

\begin{proposition}[Equivalent feasible control form]\label{prop:adimissible_control_form}
Let \(\hat u \Let \{\hat u_k\}_{k=0}^{T-1}\) be any feasible control for
Problem~\ref{problem:original}.
Then, for $k=0,\ldots,T-1$ and $j\in\indexset$, there exist feedback matrix \(K_k(j)\in\mathbb{R}^{n_u\times n_x}\), feedforward vector \(v_k(j)\in\mathbb{R}^{n_u}\) and matrix $V_k(j)\in\mathbb{S}_+^{n_u}$ such that for any $\controlfiltration_k$-measurable random vector $\nu_k\in\R[n_u]$, independent of $x_k$ conditional on $q_k=j$, with 
$$
\mathbb{E}[\nu_k\mid q_k=j]=0,\quad \mathbb{E}[\nu_k\nu_k^\T\mid q_k=j]=V_k(j),
$$
the control law 
\begin{align}
u_k = K_k(q_k)\bigl(x_k-\mu_k(q_k)\bigr) + v_k(q_k) + \nu_k,
\end{align}
for $k=0,\ldots,T-1$, induces the same mean trajectory \(\mu(\indexset)\), the same covariance
trajectory \(\Sigma(\indexset)\), and the same cost \(J\) as \(\hat u\).
\end{proposition}

The proof of Proposition~\ref{prop:adimissible_control_form} is given in Appendix~\ref{appendix:proof_of_admissible_control_form}. 
The next result establishes that if there are no multiplicative noise terms in~\ref{eq:plain_mjls_system}, the optimal control maintains a linear affine structure.

\begin{theorem}
Consider the dynamical system \eqref{eq:plain_mjls_system} with its associated data \ref{prob:data:1}--\ref{prob:data:3}. 
Assume $A_k(j)$ is nonsingular, and assume that, for all $r=1,\ldots,m_x$ and $s=1,\ldots,m_u$,  $A_k^{(r)}(j)=0$ and $B_k^{(s)}(j)=0$ for all $j\in\indexset$, and $k=0,\ldots,T-1$\footnote{The non-singularity assumption is standard in discrete-time models obtained from continuous-time linear dynamics over sufficiently small sampling intervals. 
For example, under a forward Euler discretization of \(\dot{x}_t=A_t(q_t)x_t+B_t(q_t)u_t\), one has \(A_k(j)=\identityMatrix_{n_x}+A_t(j)\delta t,\) which is invertible for all \(j\in\indexset\) when \(\delta t\) is sufficiently small.}. 
Then, there exists an optimal $u^\star$ for Problem~\ref{problem:original} of the form
\begin{equation}
    u_k^\star=K_k(q_k)(x_k-\mu_k(q_k))+v_k(q_k),  
    \label{eq:optimal_control_form}
\end{equation}
for $k=0,\ldots,T-1$, , where $K_k(j)\in\mathbb{R}^{n_u\times n_x}$ for $j\in\indexset$.
\label{thm:optiaml_control_form}
\end{theorem}
A proof of  Theorem \ref{thm:optiaml_control_form} can be found in Appendix~\ref{appendix:proof_of_optimal_control_form}.

\begin{remark}\label{rem:thrm1:counter:example}
We highlight the fact that Theorem~\ref{thm:optiaml_control_form} does not generally hold under multiplicative noise. 
\color{black}
Actually, one can construct a stronger counterexample where a feasible control of the form~\ref{eq:optimal_control_form} does not exist.
\color{black}
Consider the single-mode one-step system with $n_x=2$, $n_u=1$ given by
\begin{align}
    x_1=x_0+\begin{pmatrix}
        1\\
        0
    \end{pmatrix}u+\begin{pmatrix}
        0\\
        1
    \end{pmatrix}u\eta.
    \nonumber
\end{align}
Let $x_0\sim\mathcal{N}(0,\identityMatrix_2)$, $\mu_1=0$ and $\Sigma_1=2\identityMatrix_2$. 
Applying the control law $u=\nu$, 
where $\nu$ is an independent random variable with zero-mean and unit-variance, it can be verified that $\mathbb{E}[x_1]=0$ and $\mathbb{E}[x_1x_1^\T]=2\identityMatrix_2$. 
However, if we apply a control law of the form $u=K x_0+v$, where $K \Let (k_1,\;k_2)\in\mathbb{R}^{1\times 2}$ and $v\in\mathbb{R}$, the condition $\mathbb{E}[x_1]=0$ implies that $v=0$.
Then, one has $\mathbb{E}[x_1x_1^\T]=\begin{pmatrix}
    (1+k_1)^2+k_2^2 & k_2\\
    k_2 & 1+k_1^2+k_2^2
\end{pmatrix}$.
However, there are no $k_1$ and $k_2$ such that $\mathbb{E}[x_1x_1^\T]=2\identityMatrix_2$.    
\end{remark}

\begin{remark}\label{rem:comments:on:prop1}
Although Proposition~\ref{prop:adimissible_control_form} shows that, without loss of generality, one may restrict attention to linear feedback control laws augmented with an independent random term, substituting the control form~\eqref{eq:optimal_control_form} directly into Problem~\ref{problem:original} still leads to a coupling between the mean and covariance dynamics. In the absence of multiplicative noise, i.e., when \(A_k^{(r)}=0\) and \(B_k^{(s)}=0\), \cite{shrivastava2025chance} studies this problem by first solving the mean-steering subproblem and then solving the covariance-steering subproblem with the mean trajectory fixed at the values obtained from the first stage. However, this sequential procedure is not guaranteed to yield an optimal solution to the original problem. We refer the reader to~\cite{shrivastava2025chance} for further details.  
\end{remark}

%% file: sections/3-Main_Result.tex
\section{Optimal covariance steering via the lifted System}

In this section, we show that, by introducing a lifted system and by exploiting an SDP formulation, the MJLS covariance steering problem can be solved in an \emph{optimal} fashion. 
We begin with the case when no chance constraints are imposed on the state or control trajectories. 
We refer to this setting as the \emph{unconstrained} formulation, although the problem still contains certain modified equality-type constraints. 
When no ambiguity arises, we simply call it the \emph{unconstrained problem}; see \S\ref{subsec:uncon:case} below.
Next, we study the setting in which two classes of chance constraints, namely, ball constraints and half-space constraints, are incorporated. 
For this \emph{constrained covariance steering problem} (see \S\ref{subsec:constrained:case}), we derive a sufficient condition that guarantees the chance constraints are satisfied, and the ensuing optimization problem can be solved efficiently
as an SDP. 
We also establish  an iterative scheme that uses the previous iterate as a reference, with the aim of reducing the conservatism inherent in the constrained SDP formulation.

\subsection{Unconstrained case}\label{subsec:uncon:case}

%Let \(n_{\tilde x} \Let n_x+1 \in \N\), and 
% \panos{$n_x + 1$ appears in only a few places. No need to introduce new notation. I tried to fix, but I may have missed a few.}
Define the lifted state
\begin{align}\label{eq:lifted:states}
    \tilde{x}_k \Let (x_k^{\T} \;\; 1)^\T \in \R[n_{x}+1] \quad\text{for }k=0,\ldots,T.
\end{align}
Then, for all \(k=0,\ldots,T-1\), $\tilde x_k$ satisfies a lifted Markov jump linear system with multiplicative noise:
\begin{align}
    &\tilde x_{k+1}=\tilde A_k(q_k)\tilde x_k+\tilde B_k(q_k) u_k + \sum_{r=1}^{m_x} \tilde A_k^{(r)}(q_k) \tilde x_k \xi_k^{(r)} 
    \nonumber\\
    &\quad + \sum_{s=1}^{m_u} \tilde B_k^{(s)}(q_k) u_k \eta_k^{(s)} +\tilde D_k(q_k) w_k,
    \label{eq:lifted_system}
\end{align}
where for $k=0,\ldots, T-1$, $j\in\indexset$, $r=1,\ldots,m_x$, and $s=1,\ldots,m_u$, various matrices in \eqref{eq:lifted_system} are given by
\begin{align}
    &\tilde A_k(j) \Let \begin{pmatrix}
        A_k(j) & 0\\
        0 & 1
    \end{pmatrix}, \quad
    \tilde B_k(j) \Let \begin{pmatrix}
        B_k(j)\\
        0
    \end{pmatrix},
    \nonumber\\
    &\tilde A_k^{(r)}(j) \Let \begin{pmatrix}
        A_k^{(r)}(j) & 0\\
        0 & 0
    \end{pmatrix}, \quad
    \tilde B_k^{(s)}(j) \Let \begin{pmatrix}
        B_k^{(s)}(j)\\
        0
    \end{pmatrix},\,\,
    \tilde D_k(j)=\begin{pmatrix}
        D_k(j)\\
        0
    \end{pmatrix}.
    \nonumber
\end{align}
For $k=0,\ldots,T$ and $j\in\indexset$, define the \emph{conditional} second moment of $\tilde x_k$ by
\begin{align}
    &\tilde I_k(j)\coloneqq \mathbb{E}\big[\tilde x_k\tilde x_k^\T \mid q_k=j \big]=\begin{pmatrix}
        \Sigma_k(j)+\mu_k(j)\mu_k(j)^\T & \mu_k(j)\\
        \mu_k(j)^\T & 1
    \end{pmatrix},
    \label{eq:relationship_conditional_second_moment}
\end{align}
and the \emph{unconditional} second moment of $\tilde x_k$ by
\begin{align}
    &\tilde I_k:=\mathbb{E}[\tilde x_k\tilde x_k^\T]
    =\sum_{j\in\indexset}\rho_k(j)\tilde I_k(j)=\begin{pmatrix}
        \Sigma_k+\mu_k\mu_k^\T & \mu_k\\
        \mu_k^\T & 1
    \end{pmatrix}.
    \label{eq:relationship_unconditional_second_moment}
\end{align}
The cost function, in the lifted variables, can be written as
\begin{align}
    \tilde J& \Let \sum_{k=0}^{T-1}\mathbb{E}[\tilde x_k^\T \tilde Q_k(q_k)\tilde x_k+u_k^\T R_k(q_k)u_k],
    \label{eq:lifted_cost}
\end{align}
where for $k=0,\ldots,T-1$ and $j\in\indexset$, {\small$\tilde Q_k(j) \Let \begin{pmatrix}
    Q_k(j) & 0\\
    0 & 0
\end{pmatrix}$.}

For notational convenience, for $j\in\indexset$, we define the following quantities: 
\begin{enumerate}[leftmargin=*]
    \item $I_k(j) \Let \mathbb{E}[x_kx_k^\T\mid q_k=j]$, $I_k \Let \mathbb{E}[x_kx_k^\T]$, for $k=0,\ldots,T$. \label{quantities_1}
    \item $v_k(j):=\mathbb{E}[u_k\mid q_k=j]$, $Y_k(j):=\mathbb{E}[u_ku_k^\T\mid q_k=j]$, $\tilde U_k(j):=\mathbb{E}[u_k\tilde x_k^\T\mid q_k=j]$, for $k=0,\ldots,T-1$.\label{quantities_2}
    \item  $I_\mathrm{in}(j):=\Sigma_\mathrm{in}(j)+\mu_\mathrm{in}(j)\mu_\mathrm{in}(j)^\T$, and $I_\mathrm{out}:=\Sigma_\mathrm{out}+\mu_\mathrm{out}\mu_\mathrm{out}^\T$.\label{quantities_3}
\end{enumerate}
%Recalling the notation $(A)_{1:p,1:q}$ for a given matrix \(A\), defined in \S\ref{subsec:notation}, we 
We also observe that\footnote{Recall the notation \((A)_{1:p,1:q}\) in \S \ref{subsec:notation}.}
\begin{align}
    &I_k(j)=\big(\tilde I_k(j)\big)_{1:n_x,1:n_x},\,\,I_k=\big(\tilde I_k\big)_{1:n_x,1:n_x},
    \nonumber\\
    &\mu_k(j)=\big(\tilde I_k(j)\big)_{1:n_x,n_{x}+1},\,\,v_k(j)=\big(\tilde U_k(j)\big)_{1:n_u,n_{x}+1}.
    \label{eq:I_with_I_tilde}
\end{align}

\begin{remark}
The lifted-system formulation \eqref{eq:lifted_system} is central to our approach. By embedding the mode-dependent means \(\mu(\indexset)\) and covariances \(\Sigma(\indexset)\) of the original state process into the second moments \(\tilde{I}(\indexset)\) of the lifted state \(\tilde{x}\), we obtain a unified description of the statistical steering objective in terms of a single matrix-valued moment variable. In contrast to approaches that decouple the problem into separate mean- and covariance-steering components \cite{shrivastava2025chance}, this formulation retains the intrinsic coupling between these quantities within one optimization framework. This yields the lifted second-moment steering problem in Problem~\ref{problem:lifted} below. 
\end{remark}

Consider the second moment steering problem of the lifted system.
\begin{problem}[Moment covariance steering]\label{problem:lifted}
Consider the lifted dynamical system \eqref{eq:lifted_system}. 
Then, the moment covariance steering problem is given by
\begin{empheqboxed}
\begin{align}
    \min_{u\in\mathcal{U}}\quad &\tilde J = \sum_{k=0}^{T-1}\mathbb{E}[\tilde x_k^\T \tilde Q_k(q_k)\tilde x_k+u_k^\T R_k(q_k)u_k],
    \\
    \sbjto \quad &\textup{dynamics }~\ref{eq:lifted_system},\label{eq:in_lifted_formulation_dynamics}
    \\
    &\tilde I_0(j)=\tilde I_\mathrm{in}(j),\; j\in\indexset,
    \label{eq:lifted_initial_condition}
    \\
    &\tilde I_T=\tilde I_\mathrm{out},
    \label{eq:lifted_terminal_condition}
\end{align}
\label{eq:augmented_problem}
where for $j\in\indexset$,
\begin{align}
\tilde I_\mathrm{in}(j) \Let \begin{pmatrix}
        I_\mathrm{in}(j) & \mu_\mathrm{in}(j)\\
        \mu_\mathrm{in}(j)^\T  & 1
    \end{pmatrix},\quad 
    \tilde I_\mathrm{out} \Let \begin{pmatrix}
        I_\mathrm{out} & \mu_\mathrm{out}\\
        \mu_\mathrm{out} & 1
    \end{pmatrix}.
    \nonumber
\end{align}
\end{empheqboxed}
\end{problem}

\emph{Feasibility} and \emph{optimality} of controls for Problem~\ref{problem:lifted} are defined similarly to those for Problem~\ref{problem:original}. Our next result, Theorem~\ref{theorem:equivalence} below establishes that the lifted problem is equivalent to Problem~\ref{problem:original}.

\begin{theorem}\label{theorem:equivalence}
Consider Problem~\ref{problem:original} and Problem~\ref{problem:lifted}, together with their associated data and notation. Let $u \Let \aset[]{u_k}_{k=0}^{T-1}$. 
Then, $u$ is a feasible control for Problem~\ref{problem:original} if and only if it is a feasible control for Problem~\ref{problem:lifted}. 
Moreover, the corresponding costs coincide, i.e., \(J = \tilde J.\)
\end{theorem}

\begin{proof}
We begin with the reverse direction: let $u$ be a feasible control for Problem~\ref{problem:lifted}. 
From~\ref{eq:relationship_conditional_second_moment} and~\ref{eq:lifted_initial_condition}, we know that $\mu_0(j)=\mu_\mathrm{in}(j)$ and $\Sigma_0(j)=\Sigma_\mathrm{in}(j)$ for $j\in\indexset$.
Moreover, \ref{eq:relationship_unconditional_second_moment} and~\ref{eq:lifted_terminal_condition} imply that $\mu_T=\mu_\mathrm{out}$ and $\Sigma_T=\Sigma_\mathrm{out}$. 
Therefore, $u$ is also a feasible control for Problem~\ref{problem:original} and from~\ref{eq:lifted_cost}, we have $J=\tilde J$.

Conversely, let $u$ be a feasible control for Problem~\ref{problem:original}. From~\ref{eq:initial_condition} and~\ref{eq:relationship_conditional_second_moment}, we know that $\tilde I_0(j)=\tilde I_\mathrm{in}(j)$ and the expression $\tilde I_T=\tilde I_\mathrm{out}$ follows from~\ref{eq:terminal_condition} and~\ref{eq:relationship_unconditional_second_moment}. Therefore, $u$ is also a feasible control for Problem~\ref{problem:lifted}, and \ref{eq:lifted_cost} immediately gives, $\tilde J=J$. 
\end{proof}

Several corollaries of interest follow directly from Proposition~\ref{prop:adimissible_control_form}, Theorem~\ref{thm:optiaml_control_form}, and Theorem~\ref{theorem:equivalence}, which we present next. Corollary~\ref{cor:lifted_admissible_control_form} shows that any feasible control for Problem~\ref{problem:lifted} can be equivalently realized by a linear feedback control with an independent randomized term.
Corollary~\ref{cor:lifted_admissible_control_form:zero} shows that without multiplicative noise, the optimal control can be realized by a linear feedback control.

\begin{corollary}
Let \(\hat u \Let \aset[]{\hat u_k}_{k=0}^{T-1}\) be any feasible control for Problem~\ref{problem:lifted}. 
Then, for each \(k=0,\ldots,T-1\) and \(j\in\indexset\), there exist a feedback matrix \(\tilde K_k(j)\in\R[n_u\times (n_{x}+1)]\) and a matrix $V_k(j)\succeq 0$, such that, for any $\controlfiltration_k$-measurable random vector \(\nu_k\in\R[n_u]\), independent of \(\tilde x_k\) conditional on \(q_k=j\), such that
\[
\mathbb{E}[\nu_k\mid q_k=j]=0,
\quad
\mathbb{E}[\nu_k\nu_k^\T\mid q_k=j]=V_k(j),
\] the control law
\begin{align}\label{eq:lifted_admissible_control_form}
u_k = \tilde K_k(q_k)\tilde x_k + \nu_k,\quad k=0,\ldots,T-1,
\end{align}
induces the same moment trajectory \(\tilde I(\indexset)\) and the same cost \(\tilde J\) as \(\hat u\).
\label{cor:lifted_admissible_control_form}
\end{corollary}

\begin{proof}
    The proof follows immediately from Proposition~\ref{prop:adimissible_control_form} and Theorem~\ref{theorem:equivalence}, by noticing that $\tilde K_k(j)=\big(K_k(j)\;\;v_k(j)-K_k(j)\mu_k(j)\big)$ for $k=0,\ldots,T-1$ and $j\in\indexset$.
\end{proof}

\begin{corollary}\label{cor:lifted_admissible_control_form:zero}
If $A_k^{(r)}(j)=0$ and $B_k^{(s)}(j)=0$ for $r=1,\ldots,m_x$, $s=1,\ldots,m_u$, $j\in\indexset$, and $k=0,\ldots,T-1$, then there
exists an optimal $u^\star \Let \{u^{\star}_k\}_{k=0}^{T-1}$ of the form
\begin{equation}
    u_k^\star=\tilde K_k(q_k)\tilde x_k\quad \text{for } k=0,\ldots,T-1,
    \label{eq:lifted_optimal_control_form}
\end{equation}
for Problem~\ref{problem:lifted}, where $\tilde K_k(j)\in \mathbb{R}^{n_u\times (n_{x}+1)}$ for $j\in\indexset$.
\end{corollary}

The proof of Corollary~\ref{cor:lifted_admissible_control_form:zero} proceeds in exactly the same manner as the proof of Corollary~\ref{cor:lifted_admissible_control_form}, and therefore is omitted.

\begin{corollary}
    Let $u^\star$ be an optimal control for Problem~\ref{problem:lifted}. 
    Then,  $u^\star$ is also an optimal control for Problem~\ref{problem:original}, and vice versa.
    \label{cor:equivalence_optimal_control}
\end{corollary}

\begin{proof}
    It follows immediately from Theorem~\ref{theorem:equivalence}.
\end{proof}

With the equivalence of Problems \ref{problem:original} and \ref{problem:lifted} and their optimality correspondence established, we now proceed to derive a lossless relaxation of the lifted problem in terms of moment variables.

\subsection*{A lossless relaxation}

For convenience, for $k=0,\ldots,T-1$ and $i,j\in\indexset$, define 
$$
\displaystyle s_k^{ij}\Let\mathbb{P}(q_k=i\mid q_{k+1}=j)=\frac{p_k^{ij}\rho_k(i)}{\rho_{k+1}(j)}.
$$ 
In light of Assumption \ref{assumption:rho_positive}, \(s_k^{ij}\) is well-defined. Plugging control law~\ref{eq:lifted_admissible_control_form} in \ref{eq:lifted_system}, the second moment propagation is then given, for \(k=0,\ldots,T-1\) and \(j\in\indexset\),  by
\begin{align}
    &\tilde I_{k+1}(j)= \sum_{i\in\indexset}s_k^{ij}\Big(\big(\tilde A_k(i)+\tilde B_k(i)\tilde K_k(i)\big)\tilde I_k(i)\big(\tilde A_k(i)
    \nonumber\\
    & +\tilde B_k(i)\tilde K_k(i)\big)^\T+\tilde B_k(i)V_k(i)\tilde B_k(i)^\T 
    \nonumber\\
    &+\sum_{r=1}^{m_x}\tilde A_k^{(r)}(i)\tilde I_k(i) \tilde A_k^{(r)}(i)^\T+\sum_{s=1}^{m_u}\tilde B_k^{(s)}(i)\big(\tilde K_k(i)\tilde I_k(i)\tilde K_k(i)^\T 
    \nonumber\\
    & + V_k(i)\big) \tilde B_k^{(s)}(i)^\T+\tilde D_k(i)\tilde D_k(i)^\T\Big).
    \label{eq:second_moment_propagation_by_optimal_control}
\end{align}

Equation~\ref{eq:second_moment_propagation_by_optimal_control} results in nonlinear terms, such as $\tilde K_k(i)\tilde I_k(i)$ and $\tilde K_k(i)\tilde I_k(i)\tilde K_k(i)^\T$, which makes the problem nonlinear and nonconvex.
Therefore, instead of considering the actual trajectories of $\tilde x$ and $u$, that is~\ref{eq:lifted_system} and~\ref{eq:lifted_admissible_control_form}, we relax Problem~\ref{problem:lifted} by only considering the associated moment variables, $\tilde I(S)$, $\tilde U(\indexset)$ and $Y(\indexset)$, with the following conditions: for $k=0,\ldots,T-1$ and $j\in\indexset$,
\begin{align}
    &\mathbb{E}[\tilde x_k\tilde x_k^\T\mid q_k=j]=\tilde I_k(j)\succeq 0,
    \nonumber\\
    &\mathbb{E}[u_ku_k^\T\mid q_k=j]=Y_k(j)\succeq 0,
   \,\,\mathbb{E}\left[\begin{pmatrix}
        \tilde x_k\\
        u_k
    \end{pmatrix}\begin{pmatrix}
        \tilde x_k\\
        u_k
    \end{pmatrix}^\T\;\bigg|\; q_k=j\right]=\begin{pmatrix}
        \tilde I_k(j) &  \tilde U_k(j)^\T\\
         \tilde U_k(j) &  Y_k(j)
    \end{pmatrix} \succeq 0,
    \nonumber
\end{align}
as well as the propagation equations of $\tilde I(\indexset)$, obtained from~\ref{eq:lifted_system}, given by
\begin{align}
    &\tilde I_{k+1}(j)=\sum_{i\in\indexset}s_k^{ij}\Big(\tilde A_k(i)\tilde I_k(i)\tilde A_k(i)^\T +\tilde A_k(i)\tilde U_k(i)^\T  \tilde B_k(i)^\T
    \nonumber\\
    &  +\tilde B_k(i)\tilde U_k(i)\tilde A_k(i)^\T+\tilde B_k(i)  Y_k(i)\tilde B_k(i)^\T
    \nonumber\\
    &  + \sum_{r=1}^{m_x} \tilde A_k^{(r)}(i)  \tilde I_k(i) \tilde A_k^{(r)}(i)^\T + \sum_{s=1}^{m_u} \tilde B_k^{(s)}(i) Y_k(i) \tilde B_k^{(s)}(i)^\T 
    \nonumber\\
    &+ \tilde D_k(i) \tilde D_k(i)^\T \Big)\coloneqq T_{k,j}(\tilde I_k(\indexset),\tilde U_k(\indexset),Y_k(\indexset)).
    \nonumber
    % \label{eq:lifted_relaxed_moment_propagation}
\end{align}
The cost function is written as
\begin{align}
    &\tilde J = \sum_{k=0}^{T-1}\sum_{j\in\indexset}\rho_k(j)\Big(\operatorname{tr}\big(\tilde Q_k(j)\tilde I_k(j)\big)+\operatorname{tr}\big( R_k(j)Y_k(j)\big)\Big)
     \teL L\big(\tilde I(\indexset),Y(\indexset)\big).\nonumber
\end{align}
To avoid  notational ambiguity, we employ $\bar I(\indexset) \Let \{\bar I_k(j)\}_{k=0}^T$, $\bar U(\indexset) \Let \{\bar U_k(j)\}_{k=0}^{T-1}$ and $\bar Y(\indexset) \Let \{\bar Y_k(j)\}_{k=0}^{T-1}$ to substitute $\tilde I(\indexset)$, $\tilde U(\indexset)$ and $Y(\indexset)$ as decision variables and relax Problem~\ref{problem:lifted} as
\begin{empheqboxed}  
\begin{align}
   & \min_{\substack{\bar I(\indexset), \bar U(\indexset), \\
    \bar Y(\indexset)}}  L(\bar I(\indexset), \bar Y(\indexset)),\label{eq:cost_in_relaxed}
    \\
 & \bar I_k(j)\succeq 0,\; \bar Y_k(j)\succeq 0,\; \begin{pmatrix}
        \bar I_k(j) &  \bar U_k(j)^\T\\
         \bar U_k(j) & \bar Y_k(j)
    \end{pmatrix} \succeq 0,\label{eq:guarantee_loessless_in_relaxed}
    \\
    &\bar I_{k+1}(j)=T_{k,j}(\bar I_k(\indexset),\bar U_k(\indexset),\bar Y_k(\indexset)),\label{eq:I_propagation_in_relaxed}
    \\
    &\bar I_0(j)=\tilde I_\mathrm{in}(j),\; j\in\indexset,\label{eq:initial_condition_in_relaxed}
    \\
    &\bar I_T\Let\sum_{j\in\indexset}\rho_T(j)\bar I_T(j)=\tilde I_\mathrm{out},\label{eq:terminal_condition_in_relaxed}
\end{align}
where~\ref{eq:guarantee_loessless_in_relaxed} and~\ref{eq:I_propagation_in_relaxed} hold for all $k=0,\ldots,T-1$ and $j\in\indexset$.
\end{empheqboxed}

\begin{remark}\label{rem:on:lossless:relaxation}
The preceding relaxation of Problem~\ref{problem:lifted} to~\eqref{eq:cost_in_relaxed}--\eqref{eq:terminal_condition_in_relaxed} is indeed \emph{lossless} and a proof is given in in Appendix~\ref{appen:aux:lemmas}; see Lemma~\ref{lemma:guarantee_lossless}. We show that every feasible solution of~\eqref{eq:cost_in_relaxed}--\eqref{eq:terminal_condition_in_relaxed} can be realized by a feasible control for Problem~\ref{problem:lifted}. 
Therefore, Problem~\ref{problem:lifted} admits an exact SDP reformulation given by~\eqref{eq:cost_in_relaxed}--\eqref{eq:terminal_condition_in_relaxed}.
\end{remark}

\begin{remark}
If the terminal covariance condition in \eqref{eq:terminal_condition_in_relaxed} is relaxed to satisfy an upper bound on the covariance, or is imposed mode-wise, the terminal lifted constraint is modified by replacing the equality constraint with the corresponding block equality/inequality conditions.
For the unconditional case, replacing $\Sigma_T \preceq \Sigma_\mathrm{out}$ amounts to
\begin{align}
    &\bigl(\bar I_T\bigr)_{1:n_x, 1:n_x} \preceq \bigl(\tilde I_\mathrm{out}\bigr)_{1:n_x,1:n_x},\,\,
    \bigl(\bar I_T\bigr)_{1:n_x, n_x+1} = \bigl(\tilde I_\mathrm{out}\bigr)_{1:n_x,n_x+1}.
    \nonumber
\end{align}
For mode-conditioned terminal specifications, the same structure applies for each $j\in\indexset$, namely,
\begin{align}
    \bar I_T(j) = \tilde I_\mathrm{out}(j),\nonumber
\end{align}
for equality constraints, and
\begin{align}
    &\bigl(\bar I_T(j)\bigr)_{1:n_x, 1:n_x} \preceq \bigl(\tilde I_\mathrm{out}(j)\bigr)_{1:n_x,1:n_x}, \,\,\bigl(\bar I_T(j)\bigr)_{1:n_x, n_x+1} = \bigl(\tilde I_\mathrm{out}(j)\bigr)_{1:n_x,n_x+1},\nonumber
\end{align}
for covariance upper bounds.
These substitutions preserve the exactness of the relaxation and the convexity of the resulting SDP.    
\end{remark}
%\end{remark}

We now move on to the constrained setting. Specifically, we focus on two types of chance constraints imposed on the state and control trajectories: ball constraints and half-space constraints.
We focus on the constraints at fixed time steps $\constridx\in\{0,\ldots,T-1\}$, derive sufficient conditions for these constraints, and note that the results extend naturally to other time steps.
% \fangji{I changed $k\in\{0,\ldots,T-1\}$ to $s\in\{0,\ldots,T-1\}$ in order not to abuse the notations. Sid, can you help correct it if you find anywhere unchanged?} \sid{Sure, I'll take another look at the afternoon.}

\subsection{Chance-constrained case}\label{subsec:constrained:case}

We begin with the analysis of the ball-constrained problem. 
\subsubsection{Ball constraints}\label{sec:ball_constraints}
Let \(\eps_x,\eps_u \in \lcrc{0}{1}\) be user-specified violation tolerances. 
Suppose that $a_\constridx\in\mathbb{R}^{n_x}$, $b_\constridx\in\mathbb{R}^{n_u}$, and $r_{\constridx,x}> 0$ and $r_{\constridx,u}> 0$. Consider the ball constraints on $x_\constridx$ and $u_\constridx$ as follows
\begin{align}
    &\mathbb{P}\bigl(\|x_\constridx-a_\constridx\| \le r_{\constridx,x}\bigr)\ge 1-\varepsilon_x,
    \label{eq:ball_constraints_on_x}
    \\
    &\mathbb{P}\bigl(\|u_\constridx-b_\constridx\| \le r_{\constridx,u}\bigr)\ge 1-\varepsilon_u.
    \label{eq:ball_constraints_on_u}
\end{align}
To obtain tractable SDP constraints, we formulate the following sufficient conditions for~\ref{eq:ball_constraints_on_x} and~\ref{eq:ball_constraints_on_u}.

\begin{proposition}
Consider the inequalities \eqref{eq:ball_constraints_on_x}--\eqref{eq:ball_constraints_on_u} along with their associated data. For $j\in\indexset$, let $a_\constridx^\mathrm{ref}(j)\in\mathbb{R}^{n_x}$ and $b_\constridx^\mathrm{ref}(j)\in\mathbb{R}^{n_u}$ be such that
 \begin{align}
    \|a_\constridx^\mathrm{ref}(j)-a_\constridx\|<r_{\constridx,x},\quad \|b_\constridx^\mathrm{ref}(j)-b_\constridx\|< r_{\constridx,u}.\label{eq:basic_constraint_for_a_ref_and_b_ref}
\end{align}
Then,~\ref{eq:ball_constraints_on_x} is implied by
\begin{align}
    &\sum_{j\in\indexset}\frac{\rho_\constridx(j)}{(r_{\constridx,x}-\|a_\constridx^\mathrm{ref}(j)-a_\constridx\|)^2}\Big(\operatorname{tr}\big(I_\constridx(j)\big)-2a_\constridx^\mathrm{ref}(j)^\T \mu_\constridx(j)
    \nonumber\\
    & +a_\constridx^\mathrm{ref}(j)^\T a_\constridx^\mathrm{ref}(j))\Big)\le \varepsilon_x,
    \label{eq:convexified_ball_constraints_for_x}
\end{align}
and~\ref{eq:ball_constraints_on_u} is implied by
\begin{align}
    & \sum_{j\in\indexset}\frac{\rho_\constridx(j)}{(r_{\constridx,u}-\|b_\constridx^\mathrm{ref}(j)-b_\constridx\|)^2}\Big(\operatorname{tr}\big(Y_\constridx(j)\big)-2b_\constridx^\mathrm{ref}(j)^\T v_\constridx(j)
     \nonumber\\
     & +b_\constridx^\mathrm{ref}(j)^\T b_\constridx^\mathrm{ref}(j)\Big)\le \varepsilon_u.\label{eq:convexified_ball_constraints_for_u}
\end{align}
\label{proposition:ball_constraints}
\end{proposition}
A proof of Proposition~\ref{proposition:ball_constraints} is provided in Appendix~\ref{appendix:proof_ball_constraints}.

\begin{remark}
The reference variables \(a_\constridx^\mathrm{ref}(\indexset)\) and \(b_\constridx^\mathrm{ref}(\indexset)\) are free design parameters, subject only to~\eqref{eq:basic_constraint_for_a_ref_and_b_ref}. A straightforward choice is \(a_\constridx^\mathrm{ref}(j) \Let a_\constridx\) and \(b_\constridx^\mathrm{ref}(j) \Let b_\constridx\) for all \(j\in\indexset\), but this can yield rather conservative results. To reduce conservatism, we introduce a systematic iterative scheme for updating the reference variables, as described in \S\ref{sec:iterative_framework}.
\end{remark}

\subsubsection{Halfspace constraints}
\label{sec:polytope_constraints}

Let \(\eps_x,\eps_u \in \lcrc{0}{1}\) be user-specified violation tolerances. 
Suppose that $\alpha_{\constridx,x}\in\mathbb{R}^{n_x}$, $\alpha_{\constridx,u}\in\mathbb{R}^{n_u}$, $\beta_{\constridx,x} \in \R,$ and $\beta_{\constridx,u}\in\mathbb{R}$, and consider the half-space constraints
\begin{align}
    &\mathbb{P}(\alpha_{\constridx,x}^\T x_\constridx\le \beta_{\constridx,x})\ge 1-\varepsilon_x,
    \label{eq:polytope_constraints_on_x}
    \\
    &\mathbb{P}(\alpha_{\constridx,u}^\T u_\constridx\le \beta_{\constridx,u})\ge 1-\varepsilon_u.
    \label{eq:polytope_constraints_on_u}
\end{align}
We formulate the following sufficient conditions for~\ref{eq:polytope_constraints_on_x} and~\ref{eq:polytope_constraints_on_u}.

\begin{proposition}
Consider the inequalities \eqref{eq:polytope_constraints_on_x}--\eqref{eq:polytope_constraints_on_u} along with their associated data.
Let $\mu_\constridx^\mathrm{ref}(j)\in\mathbb{R}^{n_x}$ and $v_\constridx^\mathrm{ref}(j)\in\mathbb{R}^{n_u}$ for $j\in\indexset$.
Then \ref{eq:polytope_constraints_on_x} is implied by,  for all $j\in\indexset$,
\begin{align}
    &(1-\varepsilon_x)\alpha_{\constridx,x}^\T I_\constridx(j)\alpha_{\constridx,x} - 2\big(\mu_\constridx^\mathrm{ref}(j)^\T\alpha_{\constridx,x}
    \nonumber\\
    &\quad -\varepsilon_x\beta_{\constridx,x}\big)\alpha_{\constridx,x}^\T \mu_\constridx(j)\le \varepsilon_x\beta_{\constridx,x}^2-(\alpha_{\constridx,x}^\T\mu_\constridx^\mathrm{ref}(j))^2,\label{eq:convexified_polytope_constraints_on_x_1}
    \\
    &\alpha_{\constridx,x}^\T\mu_\constridx(j)\le \beta_{\constridx,x},\label{eq:convexified_polytope_constraints_on_x_2}
\end{align}
and~\ref{eq:polytope_constraints_on_u} is implied by, for all $j\in\indexset$,
\begin{align}
    &(1-\varepsilon_u)\alpha_{\constridx,u}^\T Y_\constridx(j)\alpha_{\constridx,u} - 2\big(v_\constridx^\mathrm{ref}(j)^\T\alpha_{\constridx,u}
    \nonumber\\
    &\quad -\varepsilon_u\beta_{\constridx,u}\big)\alpha_{\constridx,u}^\T v_\constridx(j)\le \varepsilon_u\beta_{\constridx,u}^2-(\alpha_{\constridx,u}^\T v_\constridx^\mathrm{ref}(j))^2,\label{eq:convexified_polytope_constraints_on_u_1}
    \\
    &\alpha_{\constridx,u}^\T v_\constridx(j)\le \beta_{\constridx,u}.\label{eq:convexified_polytope_constraints_on_u_2}
\end{align}
\label{proposition:polytope_constraints}
\end{proposition}
A proof of Proposition~\ref{proposition:polytope_constraints} is deferred to Appendix~\ref{appendix:proof_polytope_constraints}, while the procedure for selecting the reference variables $\mu_\constridx^\mathrm{ref}(\indexset)$ and $v_\constridx^\mathrm{ref}(\indexset)$ is described in \S\ref{sec:iterative_framework}. We are now ready to establish an SDP formulation for the problem in the presence of chance constraints.

\subsubsection{SDP formulation}

Using~\ref{eq:I_with_I_tilde}, the \emph{state} constraints in~\ref{eq:convexified_ball_constraints_for_x} and~\ref{eq:convexified_polytope_constraints_on_x_1}-\ref{eq:convexified_polytope_constraints_on_x_2} can be uniformly written as
\begin{align}
\sum_{j\in \indexset}\operatorname{tr}\big(E_{\constridx,i}(j)\tilde I_\constridx(j)\big)\le c_{\constridx,i}, 
\label{eq:final_constraints_on_I}
\end{align}
for $i=1,\dots,\ell_\constridx$,
and the \emph{control} constraints in~\ref{eq:convexified_ball_constraints_for_u} and~\ref{eq:convexified_polytope_constraints_on_u_1}-\ref{eq:convexified_polytope_constraints_on_u_2} can be uniformly written as
\begin{align}
&\sum_{j\in S}\Big(\operatorname{tr}\big(F_{\constridx,i}(j)Y_\constridx(j)\big)
+\operatorname{tr}\!\big(G_{\constridx,i}(j)\tilde U_\constridx(j)\big)\Big)\le d_{\constridx,i},
\label{eq:final_constraints_on_U_Y}
\end{align}
for \( i=1,\dots,t_\constridx\).
Here, $\ell_\constridx$ and $t_\constridx$ denote the total number of state and control constraints, respectively, the index $i$ enumerates the corresponding constraints, $E_{\constridx,i}(j), F_{\constridx,i}(j), G_{\constridx,i}(j)$ are the coefficient matrices obtained by rewriting the inequalities in affine trace form, and $c_{\constridx,i}, d_{\constridx,i}$ are the corresponding scalar bounds. 
The exact expressions for the tuple
$E_{\constridx,i}(\indexset), F_{\constridx,i}(\indexset), G_{\constridx,i}(\indexset), c_{\constridx,i}, d_{\constridx,i}$
are provided in Appendix~\ref{appen:expressions}.

From Proposition~\ref{proposition:ball_constraints} and Proposition~\ref{proposition:polytope_constraints}, Problem~\ref{problem:lifted} with ball and/or half-space chance constraints admits the following conservative reformulation:
\begin{align}
    \min_{u\in \mathcal{U}} \quad & \ref{eq:lifted_cost},
\label{eq:relaxed_constrained_problem_begin}
\\
\text{s.t.}\hspace{1em}
& \ref{eq:in_lifted_formulation_dynamics}\text{--}\ref{eq:lifted_terminal_condition},\;\ref{eq:final_constraints_on_I}\text{--}\ref{eq:final_constraints_on_U_Y}.
\label{eq:relaxed_constrained_problem_end}
\end{align}
% Moreover, feasibility of \eqref{eq:relaxed_constrained_problem_begin}--\eqref{eq:relaxed_constrained_problem_end} \sid{What do you mean by feasibility of 41-42? I don't recommend writing this.} implies satisfaction of the original chance constraints. 
Applying the same lossless moment relaxation as in the unconstrained case yields the SDP
\begin{empheqboxed}
\begin{align}
\min_{\substack{\bar I(\indexset), \bar U(\indexset),\\ \bar Y(\indexset)}} \quad &\ref{eq:cost_in_relaxed},
\label{eq:lossless_relaxed_constrained_problem_begin}
\\
\text{s.t.}\hspace{2em}  &\ref{eq:guarantee_loessless_in_relaxed}\text{--}\ref{eq:terminal_condition_in_relaxed}, 
\\
&\hspace*{-4em}\sum_{j\in\indexset}\operatorname{tr}\big(E_{\constridx,i}(j)\bar I_\constridx(j)\big)\le c_{\constridx,i} \text{ for }i=1,\dots,\ell_\constridx,
\\
&\hspace*{-4em}\sum_{j\in S}\Big(\operatorname{tr}\big(F_{\constridx,i}(j)\bar Y_\constridx(j)\big)
+\operatorname{tr}\!\big(G_{\constridx,i}(j)\bar U_\constridx(j)\big)\Big)
\le d_{\constridx,i}
\nonumber\\
&\quad \ \text{ for } i=1,\ldots,t_\constridx.\label{eq:lossless_relaxed_constrained_problem_end}
\end{align}
\end{empheqboxed}

\begin{proof}
The proof follows immediately from Lemma~\ref{lemma:guarantee_lossless}.
\end{proof}

\color{black}

\subsection{An iterative framework with tolerance relaxation}
\label{sec:iterative_framework}

The sufficient conditions developed in Sections~\ref{sec:ball_constraints} and~\ref{sec:polytope_constraints} depend on the reference variables $a_\constridx^{\mathrm{ref}}(\indexset)$, $b_\constridx^{\mathrm{ref}}(\indexset)$, $\mu_\constridx^{\mathrm{ref}}(\indexset)$, and $v_\constridx^{\mathrm{ref}}(\indexset)$.
Although these sufficient conditions yield an SDP whose solution is feasible for the original chance-constrained problem, their tightness depends strongly on the choice of these reference variables. 
In particular, poor reference values may lead to overly conservative bounds, and may even render the resulting SDP~\eqref{eq:lossless_relaxed_constrained_problem_begin}--\eqref{eq:lossless_relaxed_constrained_problem_end}
infeasible when the prescribed violation tolerances are small.

The preceding observation 
motivates an outer iterative procedure for the choice of reference variables.
Starting from an initial choice of reference variables, we solve the SDP obtained in Section~\ref{sec:ball_constraints} and~\ref{sec:polytope_constraints}, use the returned moments and conditional means to construct improved reference values, and then resolve the SDP with the updated references. Since the bounds in Proposition~\ref{proposition:ball_constraints} and Proposition~\ref{proposition:polytope_constraints}, become tighter when the reference variables are chosen closer to the returned state and control statistics, this procedure is intended to reduce conservatism across iterations. 
Moreover, because the SDP~\eqref{eq:lossless_relaxed_constrained_problem_begin}--\eqref{eq:lossless_relaxed_constrained_problem_end} induced by the initial reference values may be infeasible for the desired tolerances, we incorporate a tolerance-relaxation mechanism. Specifically, we first solve relaxed subproblems with enlarged violation tolerances and then gradually decrease these tolerances toward their target values. This yields the iterative scheme summarized in Algorithm \(\algomjls\).

Specifically, in the proof of Proposition~\ref{proposition:ball_constraints} and Proposition~\ref{proposition:polytope_constraints}, we showed that incorporating $a_\constridx^\mathrm{ref}(\indexset)$ and $b_\constridx^\mathrm{ref}(\indexset)$ allows us to upper bound the  conditional violation probabilities $\mathbb{P}(\|x_\constridx-a_\constridx\|>r_{\constridx,x}\mid q_\constridx=j)$ and $\mathbb{P}(\|u_\constridx-b_\constridx\|>r_{\constridx,u}\mid q_\constridx=j)$ by
\begin{align}
    &\sum_{j\in\indexset}\rho_k(j)\frac{\mathbb{E}\big[\|x_\constridx-a_\constridx^\mathrm{ref}(j)\|^2\mid q_\constridx=j\big]}{(r_{\constridx,x}-\|a_\constridx^\mathrm{ref}(j)-a_\constridx\|)^2}
     \label{eq:upper_bound_of_ball_probabilitiesA}
\end{align}
and
\begin{align}
 &\sum_{j\in\indexset}\rho_k(j)\frac{\mathbb{E}\big[\|u_\constridx-b_\constridx^\mathrm{ref}(j)\|^2\mid q_\constridx=j\big]}{(r_{\constridx,u}-\|b_\constridx^\mathrm{ref}(j)-b_\constridx\|)^2},
    \label{eq:upper_bound_of_ball_probabilitiesB}
\end{align}
respectively.
Applying Lemma~\ref{lemma:ball_reference}, the analytical minimizers of~\ref{eq:upper_bound_of_ball_probabilitiesA}-\ref{eq:upper_bound_of_ball_probabilitiesB} is achieved by setting the reference 
\begin{align}
    a_\constridx^{\mathrm{ref}}(j)=\begin{cases}
        a_\constridx+
        \left(\frac{r_{\constridx,x}\|o_x(j)\|-c_x(j)}{r_{\constridx,x}-\|o_x(j)\|}\right)\frac{o_x(j)}{\|o_x(j)\|},\quad \text{if } \mathcal{T}_x \text{ holds},
        \\[12pt]
        a_\constridx,  \qquad\text{otherwise,}
    \end{cases}
    \label{eq:a_ref}
    \\
    b_\constridx^{\mathrm{ref}}(j)=\begin{cases}
        b_\constridx+
        \left(\frac{r_{\constridx,u}\|o_u(j)\|-c_u(j)}{r_{\constridx,u}-\|o_u(j)\|}\right)\frac{o_u(j)}{\|o_u(j)\|},\quad \text{if } \mathcal{T}_u\text{ holds},
        \\[12pt]
        b_\constridx,  \qquad\text{otherwise,}
    \end{cases}
    \label{eq:b_ref}
\end{align}
where $\mathcal{T}_x \Let \{\mu_\constridx(j)\neq a_\constridx,\, \|o_x(j)\|<r_{\constridx,x},\,c_x(j)<r_{\constridx,x}\|o_x(j)\|\}$, and $\mathcal{T}_u \Let \{v_\constridx(j)\neq b_\constridx,\, \|o_u(j)\|<r_{\constridx,u},\,c_u(j)<r_{\constridx,u}\|o_u(j)\|\}$ and the other quantities in \eqref{eq:a_ref}--\eqref{eq:b_ref} are given by $c_x(j) \Let \operatorname{tr}\big(I_\constridx(j)\big) - 2\mu_\constridx(j)^\T a_\constridx + \|a_\constridx\|^2$, $c_u(j) \Let \operatorname{tr}\big(Y_\constridx(j)\big) - 2v_\constridx(j)^\T b_\constridx + \|b_\constridx\|^2$, $o_x(j) \Let\mu_\constridx(j)-a_\constridx$, and $o_u(j) \Let v_\constridx(j)-b_\constridx$.
The $\mu_\constridx^\mathrm{ref}(\indexset)$ is used to provide a linear lower bound in~\ref{eq:polytope_proof_temp2} (see Appendix \ref{appendix:proof_polytope_constraints}). 
Note that when $\mu_\constridx^\mathrm{ref}(j)=\mu_\constridx(j)$,~\ref{eq:polytope_proof_temp2} becomes an equality. 
A similar conclusion applies to $v_\constridx^\mathrm{ref}(\indexset)$ as well. Therefore, an ideal choice for $\mu_\constridx^\mathrm{ref}(\indexset)$ and $v_\constridx^\mathrm{ref}(\indexset)$ is 
\begin{align}\label{eq:uv_ref}
\mu_\constridx^\mathrm{ref}(j)=\mu_\constridx(j) \text{  and  }v_\constridx^\mathrm{ref}(j)=v_\constridx(j).
\end{align}
Directly using the reference variables computed using~\ref{eq:a_ref}-\ref{eq:uv_ref} does not give us fixed reference variables, which is required to formulate~\ref{eq:convexified_ball_constraints_for_x}-\ref{eq:convexified_ball_constraints_for_u} and~\ref{eq:convexified_polytope_constraints_on_x_1}-\ref{eq:convexified_polytope_constraints_on_u_2} in terms of LMIs, since $\tilde I_\constridx(\indexset)$, $\tilde U_\constridx(\indexset)$, $Y_\constridx(\indexset)$, $\mu_\constridx(\indexset)$ and $v_\constridx(\indexset)$ are undetermined decision variables.
Therefore, we propose to use an iterative approach.
For the first iteration, the reference variables are initialized arbitrarily and we use these values to solve~\ref{eq:relaxed_constrained_problem_begin}-\ref{eq:relaxed_constrained_problem_end}.
In each subsequent iteration, the reference variables are computed using the $\tilde I_\constridx(\indexset)$, $\tilde U_\constridx(\indexset)$, $Y_\constridx(\indexset)$, $\mu_\constridx(\indexset)$ and $v_\constridx(\indexset)$ values obtained from the previous iteration.
This yields a convex SDP, which is solved to generate better reference variables, which, in turn, yield tighter bounds for the subsequent iteration.
Nevertheless, it is still possible that the initial reference variables may yield overly conservative bounds, potentially rendering the SDP infeasible during the first few iterations.
%
%
% \panos{The next seems repetitive. I believe you mentioned that already}
% \fangji{\red{I don't think this is mentioned anywhere before. But I made some changes. Can you maybe check the current version?}}
\color{black}
To address this issue, we adopt a sequence of scaling factors $\{\gamma^{(m)}\}_{m=1}^M$, such that $\gamma^{(1)}\ge 1$,  gradually decreasing to $\gamma^{(M)}=1$, to the violation tolerances.
The underlying rationale is that a larger tolerance relaxes the constraints~\ref{eq:final_constraints_on_I} and~\ref{eq:final_constraints_on_U_Y}, thereby increasing the likelihood of feasibility. 
Once improved reference variables are obtained, they can be used to solve the problem under smaller tolerance values.
The pseudocode for the iterative framework is provided in \(\algomjls\).
% Note that $\gamma^{(m)}=1$ is required in the stopping criteria.
\color{black}

{
\renewcommand{\algorithmcfname}{\(\textsf{MJLS-CovSteer}\)}
\renewcommand{\thealgocf}{}
\begin{algorithm2e}[t]
\SetAlgoLined
\DontPrintSemicolon

\SetKwInOut{ini}{Initialize}
\SetKwInOut{giv}{Data}

\giv{$\{\gamma^{(m)}\}_{m=1}^M$, $\varepsilon_x$, $\varepsilon_u$}
\ini{$a_\constridx^\mathrm{ref}(\indexset),\,b_\constridx^\mathrm{ref}(\indexset),\,\mu_\constridx^\mathrm{ref}(\indexset),\,v_\constridx^\mathrm{ref}(\indexset)$}

\For{$m=1,\dots,M$}{
\(\circ\) $\varepsilon_x' \leftarrow \min(\gamma^{(m)} \varepsilon_x,1)$,\quad 
$\varepsilon_u' \leftarrow \min(\gamma^{(m)} \varepsilon_u, 1)$\;

\(\circ\) Use $a_\constridx^\mathrm{ref}(\indexset),\,b_\constridx^\mathrm{ref}(\indexset),\,\mu_\constridx^\mathrm{ref}(\indexset),\,v_\constridx^\mathrm{ref}(\indexset),\,\varepsilon_x',\,\varepsilon_u'$
to solve~\eqref{eq:lossless_relaxed_constrained_problem_begin}--\eqref{eq:lossless_relaxed_constrained_problem_end}\;

\(\circ\) Update $a_\constridx^\mathrm{ref}(\indexset),\,b_\constridx^\mathrm{ref}(\indexset),\,\mu_\constridx^\mathrm{ref}(\indexset),\,v_\constridx^\mathrm{ref}(\indexset)$
using~\eqref{eq:a_ref}--\eqref{eq:uv_ref} with the returned values of
$\tilde I_\constridx(\indexset)$, $\tilde U_\constridx(\indexset)$, $Y_\constridx(\indexset)$, $\mu_\constridx(\indexset)$, and $v_\constridx(\indexset)$\;

% \If{the stopping criterion is satisfied}{
% \textbf{return} solution\;
% }
}
\textbf{return} solution\;
\caption{An iterative framework}
\label{alg:mljs:cov:steer}
\end{algorithm2e}
}

%% file: sections/4-Numerics.tex
\section{Numerical results}
\label{sec:experiments}
\begin{figure*}[t]
\centering
    \subfloat[$m=1$]{%
        \includegraphics[width=0.48\textwidth]{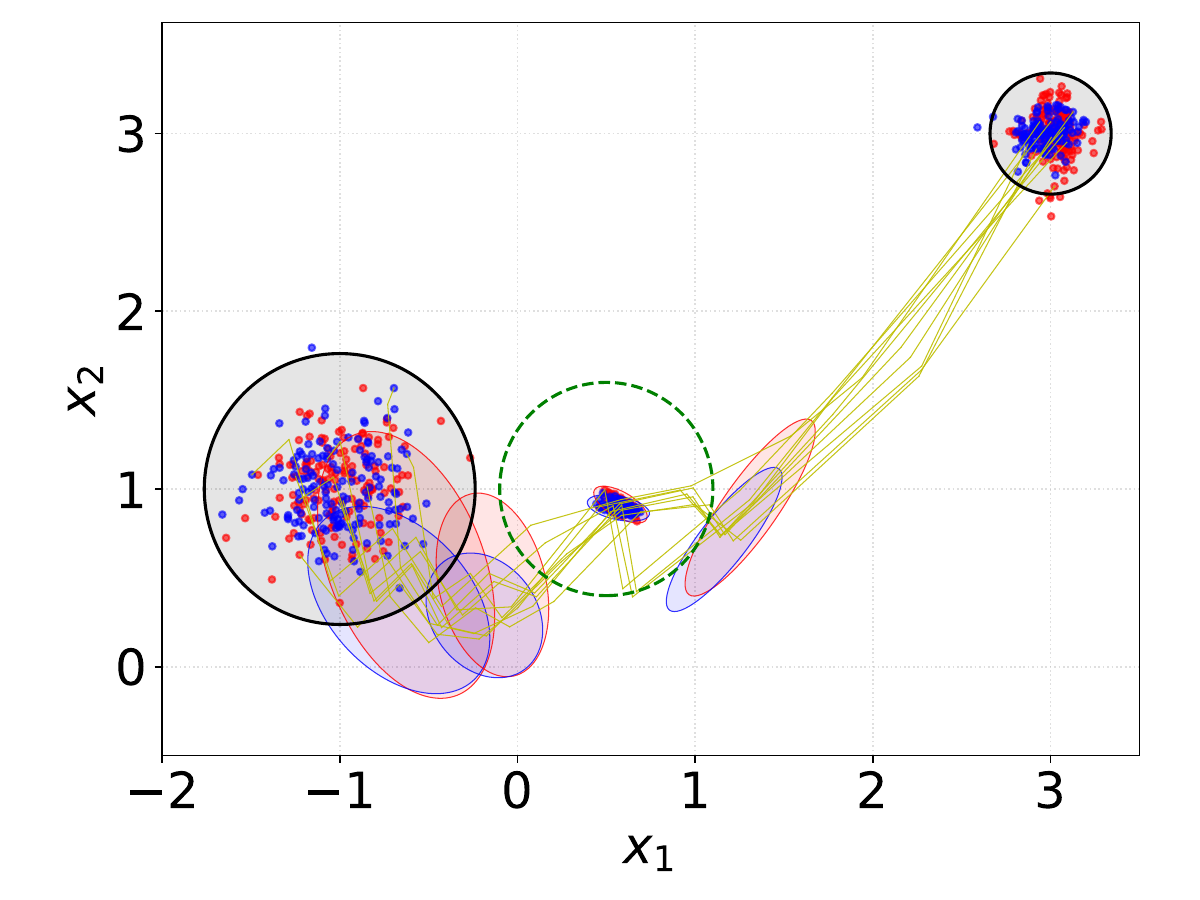}}
    \hfill
    \subfloat[$m=5$]{%
        \includegraphics[width=0.48\textwidth]{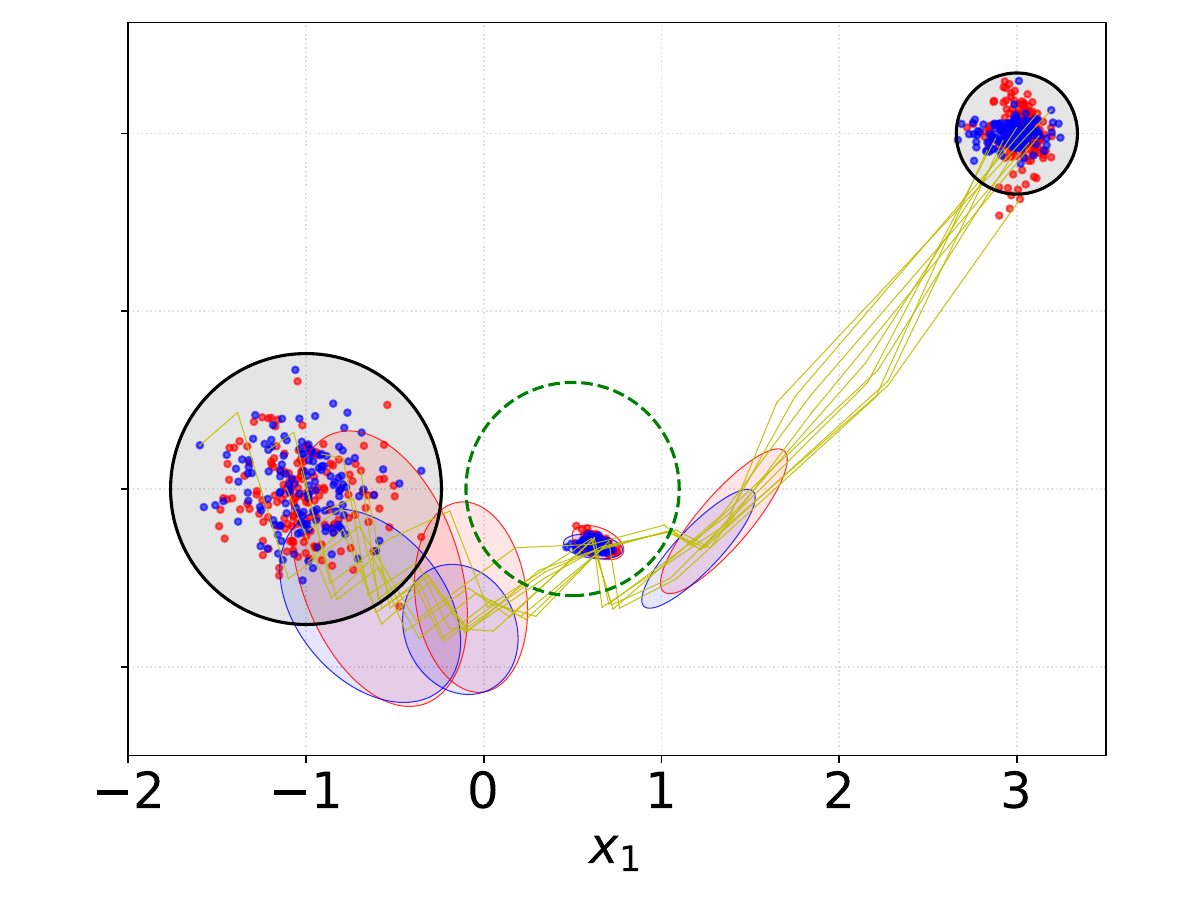}}\\[0.5em]
    \subfloat[$m=10$]{%
        \includegraphics[width=0.48\textwidth]{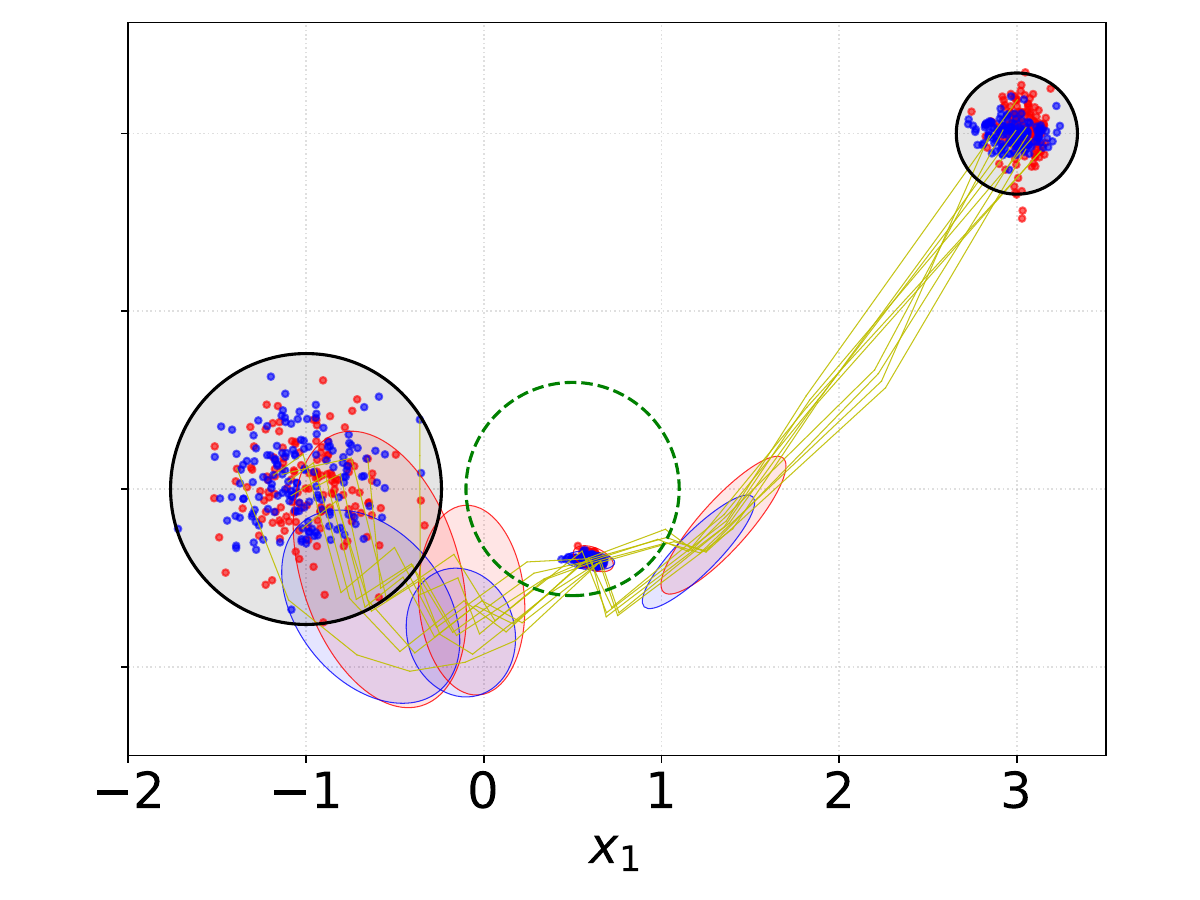}}
    \hfill
    \subfloat[Cost history]{%
        \includegraphics[width=0.48\textwidth]{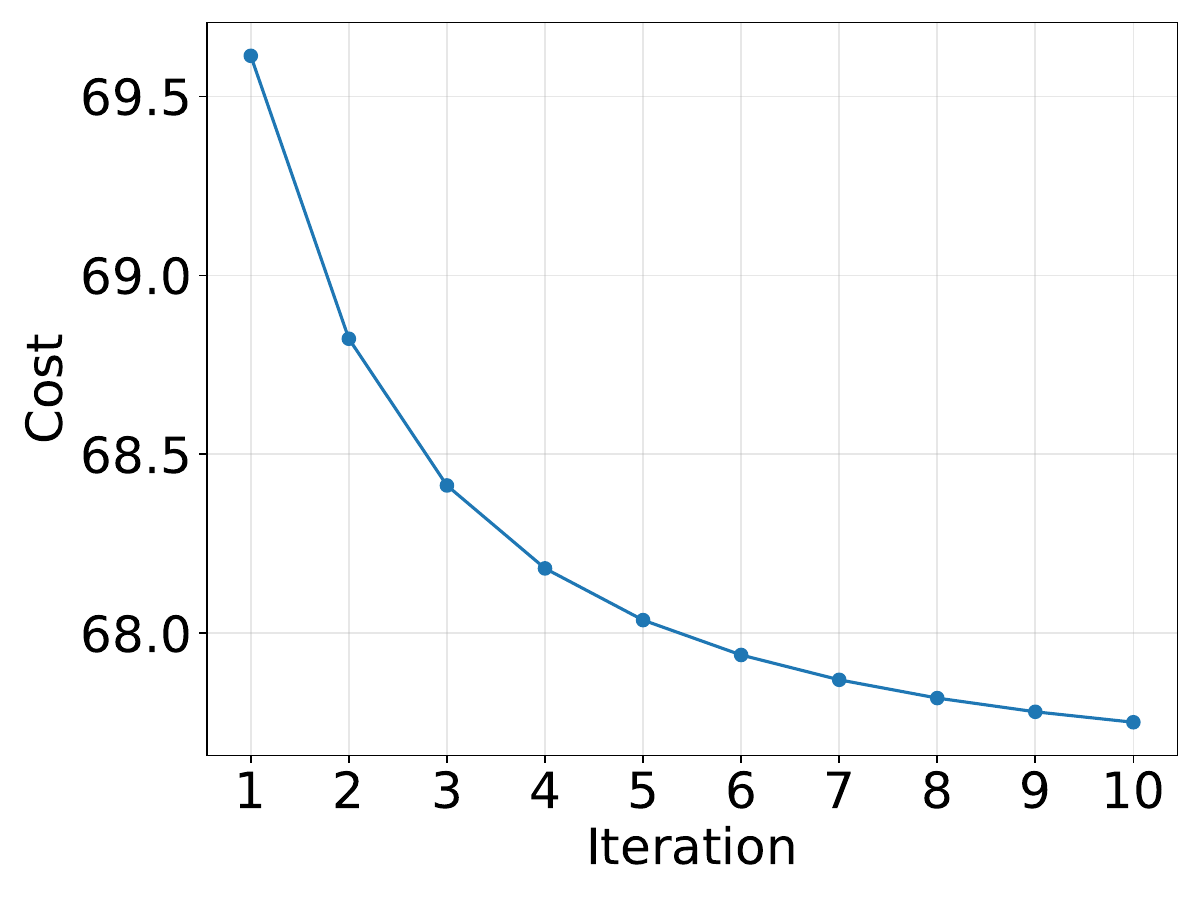}}
    \caption{Results for $\gamma^{(m)}=1$ for each $m$. In this case, we enforce $x_6$ to stay within the norm ball of radius $0.6$ centered at $(0.5,\,1)$ with high probability. (a)--(c) show the sampled state trajectories at iterations $m=1,5,10$, respectively. 
    The red and blue ellipsoids are the $0.997$-confidence ellipsoids at $k=2,4,6,8$ induced by the Gaussian distributions $\mathcal{N}(\mu_k(j),\Sigma_k(j))$, $j=1,2$, respectively; the red and blue dots are the sampled states at $k=0,6,T$ for the two modes, respectively; the yellow lines are the sampled trajectories. (d) shows the corresponding cost history over the iterations. This figure, as well as Figures~\ref{fig:reduced_gamma} and~\ref{fig:hedging_results}, is plotted by assuming a Gaussian initial distribution for each mode. This assumption is used only for illustration; the results do not rely on the exact initial distribution, provided that the initial conditional means and covariances for each mode are specified.}
    \label{fig:results_for_fixed_gamma}
\end{figure*}

\begin{figure*}[t]
\centering
    \subfloat[Unconstrained]{%
        \includegraphics[width=0.48\textwidth]{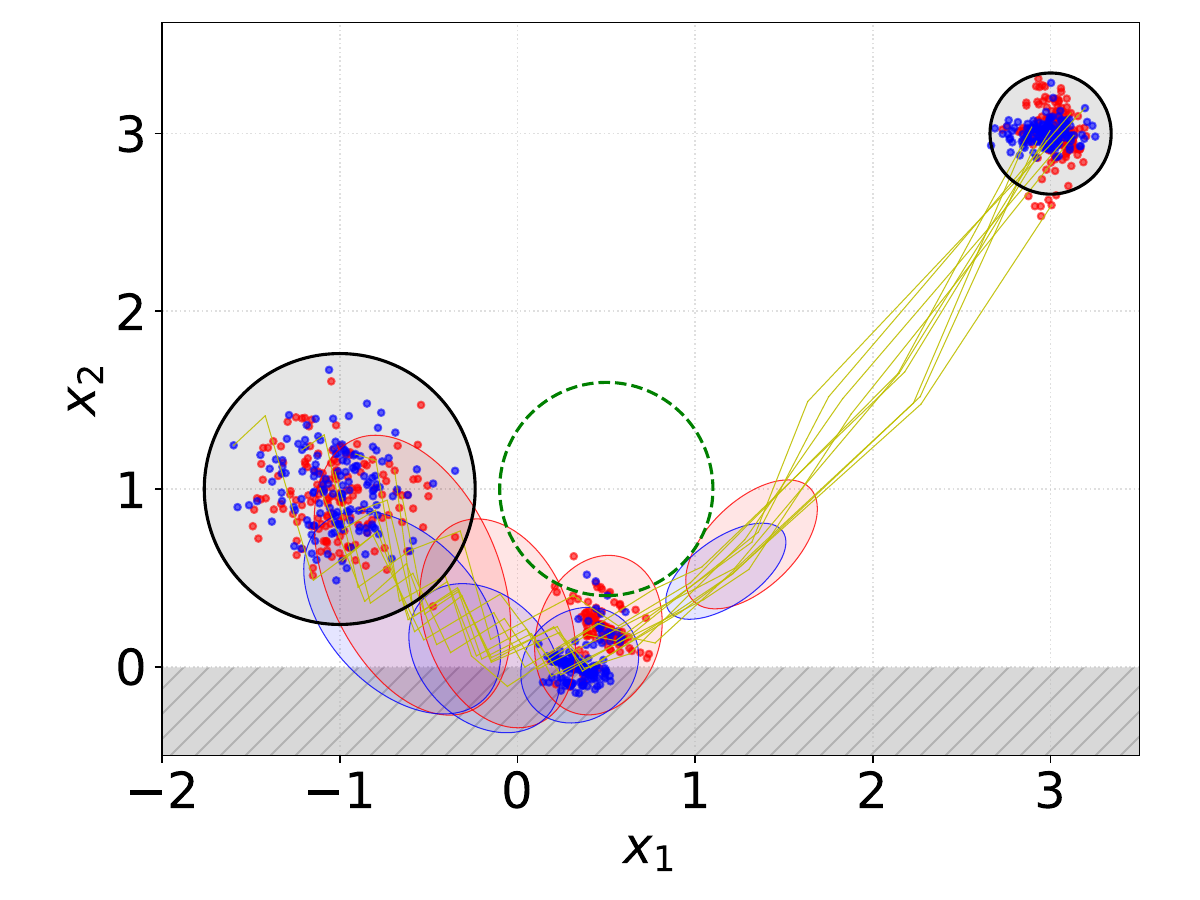}}
    \hfill
    \subfloat[Constrained]{%
        \includegraphics[width=0.48\textwidth]{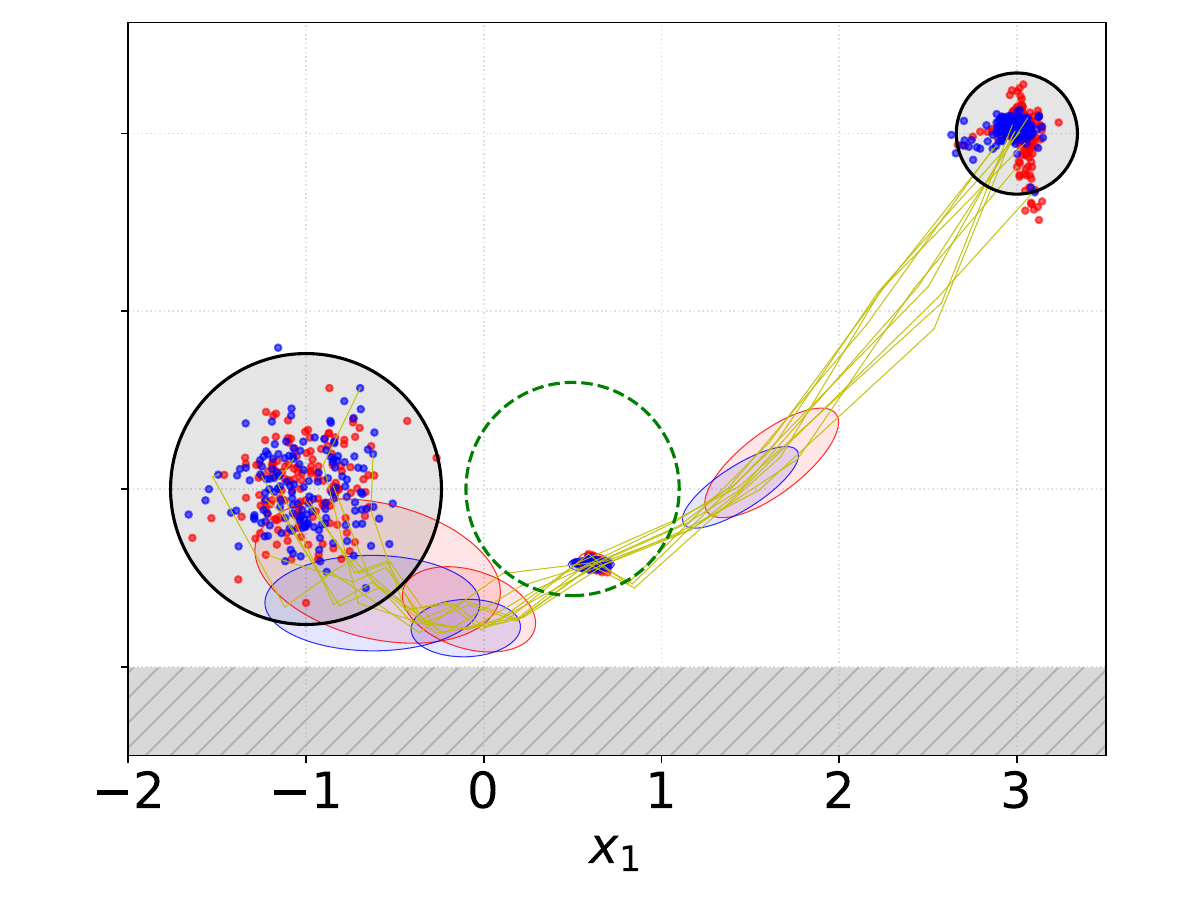}}\\[0.5em]
    \subfloat[Control profile]{%
        \includegraphics[width=0.48\textwidth]{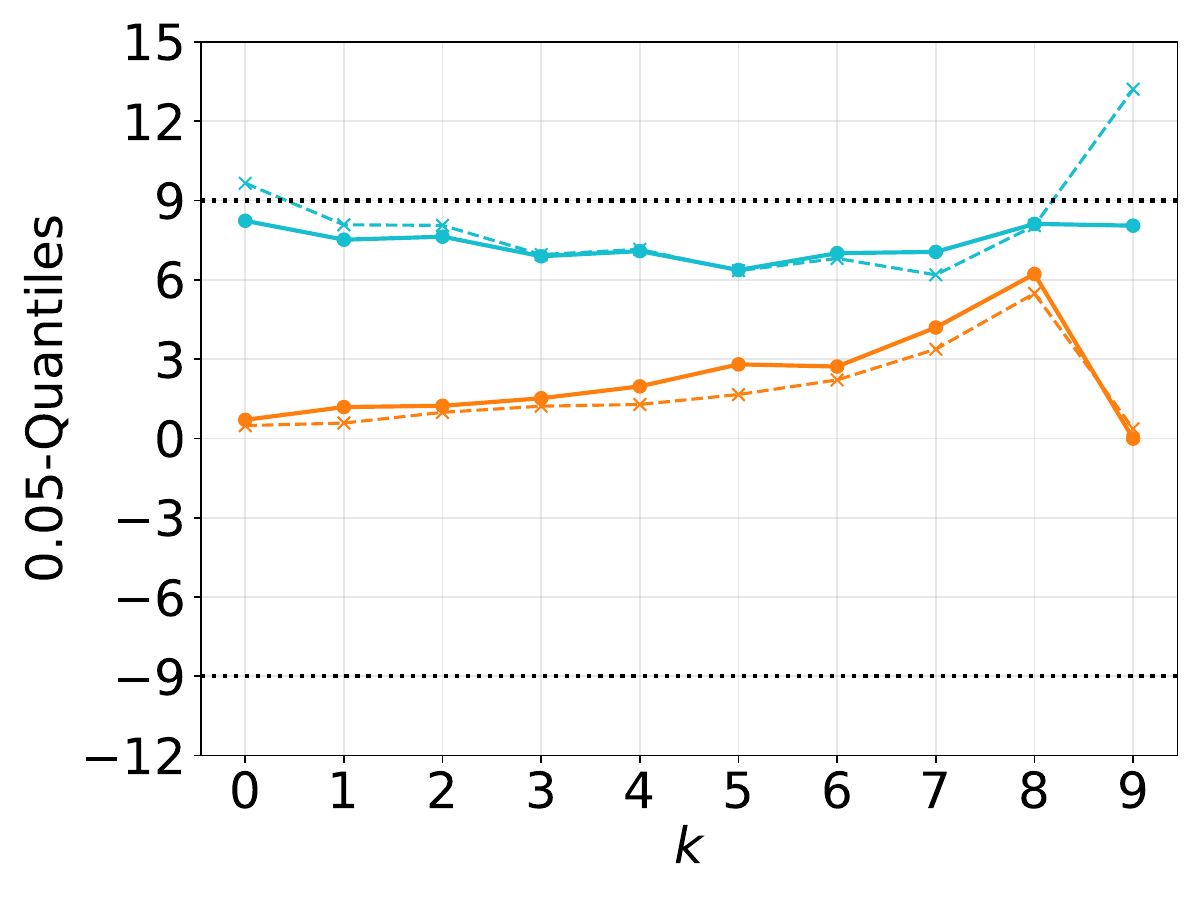}}
    \hfill
    \subfloat[Mode switch profile]{%
        \includegraphics[width=0.48\textwidth]{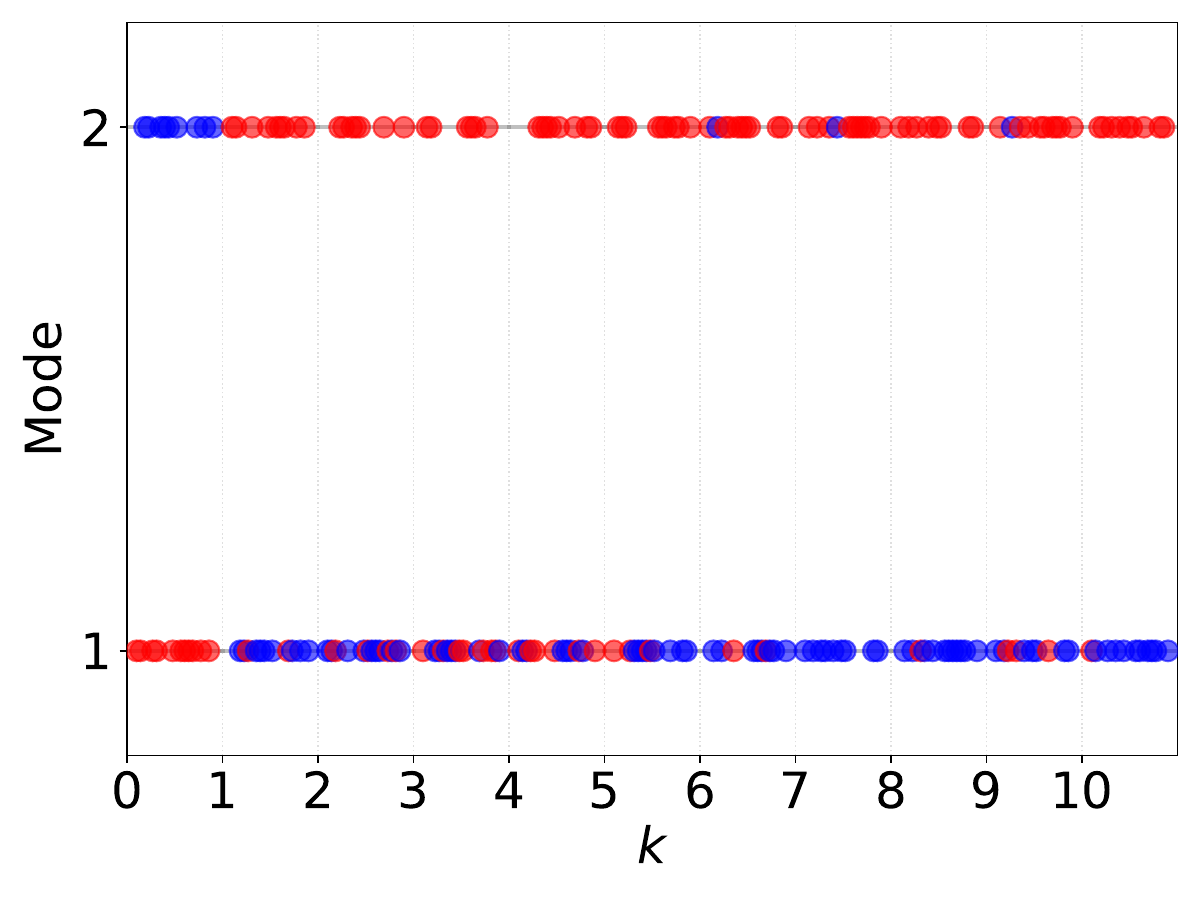}}

    \caption{Results for iteratively reduced $\gamma$-values. In this case, in addition to the ball constraint on $x_6$, we also enforced $x_{k,1}\ge 0$ and $|u_{k,1}|\le 9$ with high probability for each $k$. (a) and (b) show the sampled state trajectories for the unconstrained and constrained problems, respectively, at the last iteration. The cyan and orange lines in (c) are the upper and lower $0.05$-quantiles of the sampled $u_{k,1}$, respectively, where the solid lines correspond to the constrained case and the dashed lines correspond to the unconstrained case.
    (d) illustrates the mode switches of 20 sampled particles, where for mode 1, red dots denote particles that were in mode 1 in the previous step, while blue dots denote particles that were in mode 2 in the previous step; for mode 2, the color assignment is reversed.}
    \label{fig:reduced_gamma}
\end{figure*}

We have evaluated the performance of the proposed iterative chance-constrained framework in Section~\ref{sec:iterative_framework} under two numerical setups. The first highlights the benefit of iterative reference updates, while the second demonstrates the role of tolerance relaxation in recovering feasibility for more restrictive chance-constrained problems. In our experiments, all the reference values were initialized to zero.
% \fangji{Do we need to mention the hedging example here?}

\subsection{Chance-constrained MJLS problem}

The numerical examples use  $n_x=n_u=n_w=2$, $m_x=m_u=1$, $N_\indexset=2$, and $T=10$. For all $k=0,\ldots,T-1$, the system matrices are
\begin{align}
    &A_k(1)=0.9\,\identityMatrix_2,\quad A_k(2)=1.2\,\identityMatrix_2,
    \nonumber\\
    & B_k(1)=\begin{pmatrix}
        0.1 & 0\\
        0.2 & 0.1
    \end{pmatrix},\quad B_k(2)=\begin{pmatrix}
        0.05 & 0.1
        \\
        0 & 0.2
    \end{pmatrix},
    \nonumber\\
    &D_k(1)=D_k(2)=0.01\,\identityMatrix_2,\quad A_k^{(1)}(1)=A_k^{(1)}(2)=0.01\,\identityMatrix_2,
    \nonumber\\
    &B_k^{(1)}(1)=\begin{pmatrix}
        0.01 & 0\\
        0 & 0
    \end{pmatrix},\quad
    B_k^{(1)}(2)=\begin{pmatrix}
        0 & -0.01\\
        0 & 0.01
    \end{pmatrix}.
    \nonumber
\end{align}
The initial conditions are
\begin{align}
\mu_\mathrm{in}(1)=\mu_\mathrm{in}(2)=(-1,1)^\T,\quad \Sigma_\mathrm{in}(1)=\Sigma_\mathrm{in}(2)=0.05\,\identityMatrix_2,
    \nonumber
\end{align}
and the terminal conditions are
\begin{align}
    \mu_\mathrm{out}=(3, 3)^\T,\quad \Sigma_\mathrm{out}=0.01\,\identityMatrix_2.
    \nonumber
\end{align}
The mode transition matrix is given by
\[
(p_k^{ij})_{i,j}=
\begin{pmatrix}
    0.2 & 0.8\\
    0.9 & 0.1
\end{pmatrix},
\qquad k=0,\ldots,T-1,
\]
with initial condition $\rho_0(1)=\rho_0(2)=0.5$. 
% \sid{0.05 or 0.5?}\fangji{0.5, corrected.}
We use $Q_k(1)=Q_k(2)=\identityMatrix_2$ and $R_k(1)=R_k(2)=0.1\identityMatrix_2$ for $k=0,\ldots,T-1$.

\subsubsection{Fixed Tolerance with Iterative Reference Updates}
We first validated the benefit of iterative reference updates by considering the following ball constraint on the state:
\begin{align}
    \mathbb{P}\bigl(\| x_6-(0.5,\, 1)^{\top} \| \le 0.6\bigr)\ge 1-0.05.
    \label{eq:ball_constraints_exp}
\end{align}
We solved the problem for $M=10$ iterations, with $\gamma^{(m)}=1$ for all $m$.
The results are shown in Figure~\ref{fig:results_for_fixed_gamma}.
As the reference values are updated, the sampled states at time $k=6$ move closer to the boundary of the admissible ball, while the objective value decreases monotonically across iterations.
This indicates that the iterative procedure produces progressively tighter sufficient conditions and therefore reduces conservatism.

\subsubsection{Reducing Tolerances}

Next, we considered a more constrained setting as follows. 
In addition to the ball constraint~\eqref{eq:ball_constraints_exp}, we imposed the halfspace constraints on the state
\begin{align}
    \mathbb{P}(x_{k,1}\ge 0)\ge 1-0.05, \quad k=0,\ldots,T,
    \nonumber
\end{align}
and the halfspace constraints on the control as follows
\begin{align}
    \mathbb{P}(u_{k,1}\le 9)\ge 1-0.05\,\,\text{and}\,\,
    \mathbb{P}(u_{k,1}\ge -9)\ge 1-0.05
    \nonumber
\end{align}
for \(k=0,\ldots,T-1\). In this case, the ensuing SDP is not
feasible under the initial reference values because the resulting constraints~\eqref{eq:final_constraints_on_I}--\eqref{eq:final_constraints_on_U_Y} are overly conservative. 
To address this issue, we employed a decreasing sequence of tolerance-scaling factors, $\gamma^{(m)}=2^{6-m}$ for $m=1,\ldots,6$. Figure~\ref{fig:reduced_gamma} shows the results at the final iteration. Without enforcing these chance constraints, the sampled trajectories violate both the state and control constraints.
In contrast, the constrained solution respects the state exclusion region and keeps the sampled control quantiles within the prescribed bounds.

\subsection{Dynamic hedging with execution risk and exposure limits}

\begin{figure}[t]
    \centering
    \subfloat[Trajectory]{%
        \includegraphics[scale=0.35]{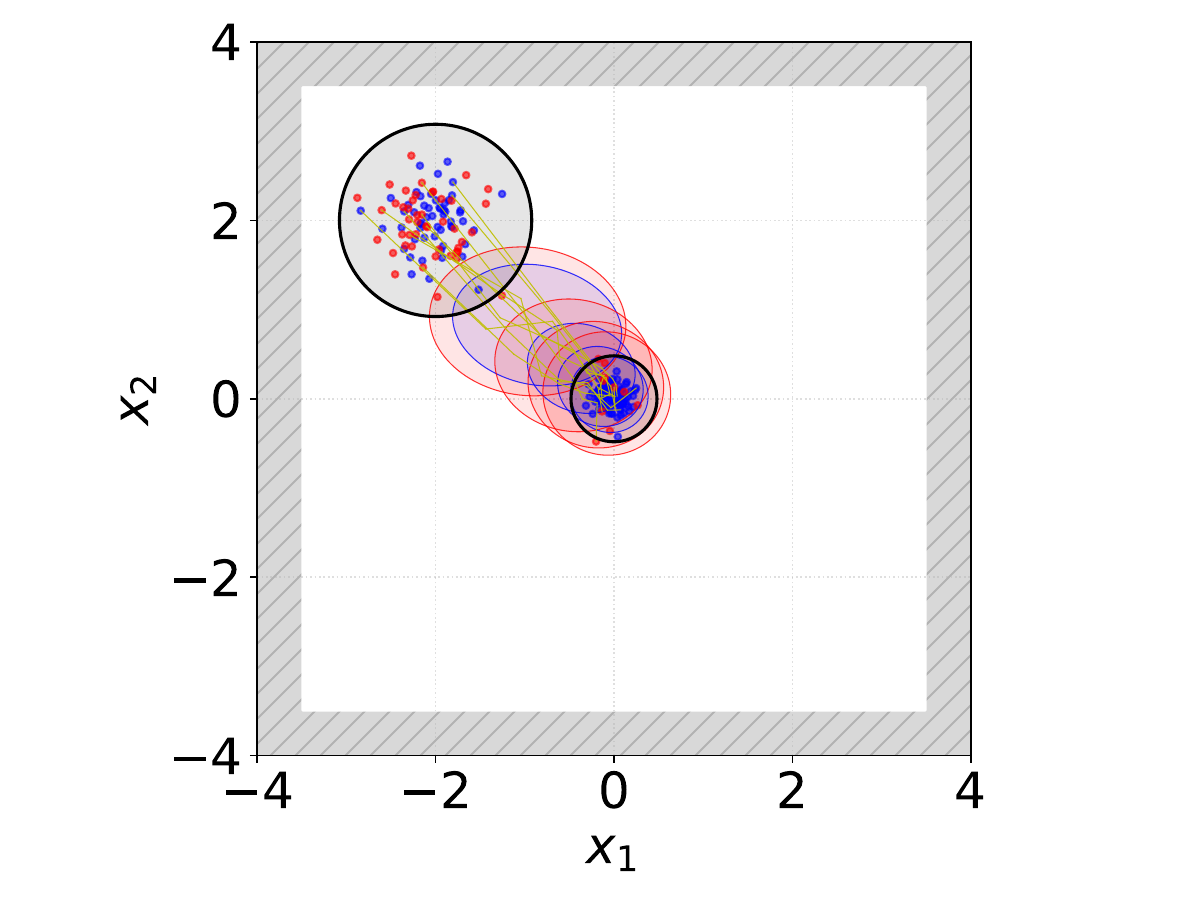}}
    \subfloat[Control profile]{%
        \includegraphics[scale=0.29]{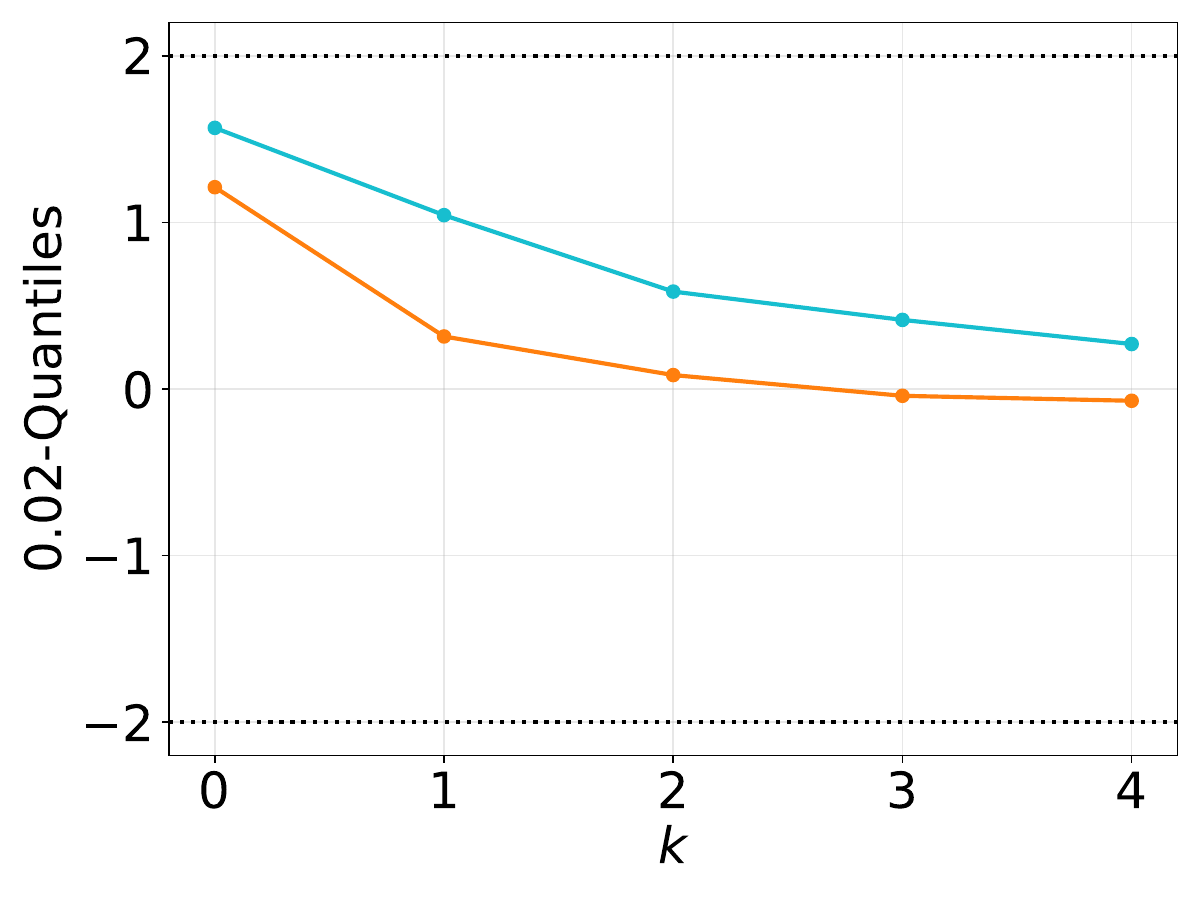}}
    \caption{Results for the hedging example.
    (a) shows sampled state trajectories together with the intermediate exposure limits.
    (b) shows the upper and lower $0.02$-quantiles of the sampled control inputs over the horizon, together with the prescribed control bounds.}
    \label{fig:hedging_results}
\end{figure}

We illustrate our framework on a dynamic hedging problem over a short horizon $T=5$. The state $x_k\in\mathbb{R}^2$ represents the tracking error in two portfolio sensitivities, where $x_{k,1}$ and $x_{k,2}$ denote the Delta and Vega errors, respectively~\cite{hull2016options}. 
The control $u_k\in\mathbb{R}^2$ represents trades in the underlying asset and in a liquid option. 
The market evolves according to a two-mode Markov chain, where Mode~1 corresponds to normal conditions and Mode~2 to distressed conditions, with a transition matrix
$$(p_k^{ij})_{i,j}=
\begin{pmatrix}
0.95 & 0.05\\
0.30 & 0.70
\end{pmatrix},\quad k=0,\ldots,T-1,$$
and initial distribution $\rho_0(1)=\rho_0(2)=0.5$.

The system parameters are chosen so that the option introduces slight Delta cross-coupling and both hedging instruments become less effective in the distressed regime.
Multiplicative noise captures state-dependent volatility ($m_x=2$) and control-dependent execution risk ($m_u=2$).
Specifically, for $k=0,\ldots,T-1$, we have
\begin{align}
    &A_k(1)=
\begin{pmatrix}
1.0 & 0.0\\
0.0 & 0.9
\end{pmatrix},\quad
B_k(1)=
\begin{pmatrix}
1.0 & 0.2\\
0.0 & 1.0
\end{pmatrix},
\nonumber\\
&A_k(2)=
\begin{pmatrix}
1.0 & 0.0\\
0.0 & 0.8
\end{pmatrix},\quad
B_k(2)=
\begin{pmatrix}
0.8 & 0.1\\
0.0 & 0.8
\end{pmatrix},
\nonumber\\
&D_k(1)=
\begin{pmatrix}
0.1 & 0.0\\
0.0 & 0.1
\end{pmatrix},\quad
D_k(2)=
\begin{pmatrix}
0.2 & 0.0\\
0.0 & 0.2
\end{pmatrix},
\nonumber\\
&A_k^{(1)}(1)=
\begin{pmatrix}
0.02 & 0.0\\
0.0 & 0.0
\end{pmatrix},\quad
A_k^{(2)}(1)=
\begin{pmatrix}
0.0 & 0.0\\
0.0 & 0.02
\end{pmatrix},
\nonumber\\
&B_k^{(1)}(1)=
\begin{pmatrix}
0.02 & 0.0\\
0.0 & 0.0
\end{pmatrix},\quad
B_k^{(2)}(1)=
\begin{pmatrix}
0.0 & 0.01\\
0.0 & 0.02
\end{pmatrix},
\nonumber\\
&A_k^{(1)}(2)=
\begin{pmatrix}
0.03 & 0.0\\
0.0 & 0.0
\end{pmatrix},\quad
A_k^{(2)}(2)=
\begin{pmatrix}
0.0 & 0.0\\
0.0 & 0.03
\end{pmatrix},
\nonumber\\
&B_k^{(1)}(2)=
\begin{pmatrix}
0.03 & 0.0\\
0.0 & 0.0
\end{pmatrix},\quad
B_k^{(2)}(2)=
\begin{pmatrix}
0.0 & 0.01\\
0.0 & 0.03
\end{pmatrix}.
\nonumber
\end{align}
The initial condition is
\[
\mu_{\mathrm{in}}(1)=\mu_{\mathrm{in}}(2)=(-2,\,2)^\top,\,\,
\Sigma_{\mathrm{in}}(1)=\Sigma_{\mathrm{in}}(2)=0.1I_2,
\]
and the objective is to drive the expected exposure to zero while enforcing
\[
\mu_{\mathrm{out}}=(0,0)^\top,\quad \Sigma_T\preceq 0.02I_2.
\]
We employed $Q_k(1)=Q_k(2)=0.1\identityMatrix_2$ and $R_k(1)=R_k(2)=0.1\identityMatrix_2$ for $k=0,\ldots,T-1$. To limit intermediate risk, for all $k=1,\dots,T-1$ we imposed
\begin{align}
    &\mathbb{P}(x_{k,1}\le 3.5)\ge 1-\varepsilon,\quad \mathbb{P}(x_{k,1}\ge -3.5)\ge 1-\varepsilon,
    \nonumber\\
    &\mathbb{P}(x_{k,2}\le 3.5)\ge 1-\varepsilon, \quad \mathbb{P}(x_{k,2}\ge -3.5)\ge 1-\varepsilon,
    \nonumber\\
    &\mathbb{P}(u_{k,1}\le 2.0)\ge 1-\varepsilon,\quad \mathbb{P}(u_{k,1}\ge -2.0)\ge 1-\varepsilon,
    \nonumber\\
    &\mathbb{P}(u_{k,2}\le 2.0)\ge 1-\varepsilon,\quad \mathbb{P}(u_{k,2}\ge -2.0)\ge 1-\varepsilon,
    \nonumber
\end{align}
with $\varepsilon=0.02$.
We employed a decreasing sequence of tolerance-scaling factors, $\gamma^{(m)}=2^{7-m}$ for $m=1,.\ldots,7$.
The resulting trajectories and control quantiles are shown in Figure~\ref{fig:hedging_results}. This example illustrates that the proposed method can hedge regime-switching Delta/Vega exposures while accounting for both volatility-scaled uncertainty and execution risk, while simultaneously enforcing probabilistic bounds on intermediate states and controls.

%% file: sections/Conclusion.tex
\section{Conclusions}
This article establishes a finite-horizon covariance steering framework for discrete-time Markov jump linear systems with both state- and control-dependent multiplicative noise. 
We showed that, in this setting, one may, without loss of generality, restrict attention to a class of control laws consisting of mode-dependent linear feedback, feedforward terms, and independent random components.
We showed that, unlike the case without multiplicative noise, a purely affine state-feedback law does not in general suffice. 
Subsequently, we introduced a lifted second-moment formulation, derived semidefinite programming reformulations for the unconstrained and chance-constrained problems, and established an iterative reference-update scheme to reduce conservatism. 

Several directions remain open for future work. On the theoretical side, it would be of interest to further sharpen the chance-constrained approximations and to better understand the conservatism introduced by the proposed surrogates and iterative updates. On the computational side, a natural direction is to develop scalable algorithms for longer horizons and larger mode sets. Extensions to partially observed MJLSs, output-feedback formulations, and more general distributional or risk-sensitive terminal specifications, along with their applications to practical scenarios, also appear to be promising directions for future investigation.

%% file: sections/appendix.tex
\section{Appendix}
In this appendix, we gather the proofs of the main results along with several supporting calculations used throughout this article. 

\subsection{Auxiliary lemmas}\label{appen:aux:lemmas}
We begin with a lemma whose proof is an immediate consequence of~\cite[Proposition 1]{liu2024reachability}.
\begin{lemma}
    Let $A\in\mathbb{S}_+^{n}$, $B\in\mathbb{S}^m$, and $C\in\mathbb{R}^{m\times n}$.
    If $\begin{pmatrix}
        A & C^\T\\
        C & B
    \end{pmatrix}\succeq 0$, then there exist $P\in\mathbb{R}^{m\times n}$ and $S\in \mathbb{S}_+^{m}$ such that $C=PA$ and $B=PAP^\T+S$.
    \label{lemma:for_lossless_proof}
\end{lemma}

The next lemma guarantees that the relaxations in \S\ref{subsec:uncon:case}--\ref{subsec:constrained:case} are lossless.

\begin{lemma}\label{lemma:guarantee_lossless}
    Let $\bar I(\indexset) \Let \aset[]{\bar I_k(j)}_{k=0}^T$, $\bar U(\indexset) \Let \{\bar U_k(j)\}_{k=0}^{T-1}$ and $\bar Y(\indexset)\Let \{\bar Y_k(j)\}_{k=0}^{T-1}$ satisfy
    \begin{align}
        &\bar I_k(j)\succeq 0,\quad \bar Y_k(j)\succeq 0,\quad \begin{pmatrix}
        \bar I_k(j) &  \bar U_k(j)^\T\\
         \bar U_k(j) & \bar Y_k(j)
    \end{pmatrix} \succeq 0,
    \nonumber\\
    &\bar I_{k+1}(j)=T_{k,j}(\bar I_k(\indexset),\bar U_k(\indexset),\bar Y_k(\indexset)),
    \nonumber
    \end{align}
    for all $k=0,\ldots,T-1$ and $j\in\indexset$.
    Then, for any initial distribution of $\tilde x_0$ with $\tilde I_0(j)=\bar I_0(j)$, there exists an admissible control $u=\{u_k\}_{k=0}^{T-1}\in\mathcal{U}$ such that the corresponding state trajectory $\{\tilde x_k\}_{k=0}^T$, generated by~\ref{eq:lifted_system},
    satisfies, for every $j\in\indexset$,
    \begin{align}
        \tilde I_k(j)=\bar I_k(j), \quad k=0,\ldots,T,
        \nonumber
    \end{align}
    and
    \begin{align}
        \tilde U_k(j)=\bar U_k(j),\; Y_k(j)=\bar Y_k(j),\quad k=0,\ldots,T-1.
        \nonumber
    \end{align}
\end{lemma}

\begin{proof}
    From Lemma~\ref{lemma:for_lossless_proof}, for each $k=0,\ldots,T-1$ and $j\in\indexset$, there exists $K_k(j)\in\R[n_u\times (n_x+1)]$ and $V_k(j)\in\mathbb{S}_+^{n_u}$ such that $\bar U_k(j)=K_k(j)\bar I_k(j)$ and $\bar Y_k(j)=K_k(j)\bar I_k(j)K_k(j)^\T+V_k(j)$.
    For $k=0,\ldots,T-1$, we sequentially construct the $\mathcal{G}_k$-measurable control
    \begin{align}
        u_k=K_k(q_k)\tilde x_k+\nu_k,
        \nonumber
    \end{align}
    where $\nu_k$ is any $\mathcal{G}_k$-measurable random vector, independent of $\tilde x_k$ conditional on $j\in\indexset$, with $\mathbb{E}[\nu_k\nu_k^\T\mid q_k=j]=V_k(j)$.
    It is then straightforward to verify that, for $j\in\indexset$, if $\tilde I_k(j)=\bar I_k(j)$, then $\tilde U_k(j)=\bar U_k(j)$ and $Y_k(j)=\bar Y_k(j)$, and the $\tilde x_{k+1}$ obtained from~\ref{eq:lifted_system} satisfies $\tilde I_{k+1}(j)=\bar I_{k+1}(j)$.
\end{proof}

% \blue{Sid: Yes, looks good.}

\begin{lemma}
\label{lemma:ball_reference}
% \sid{Here also change the \(I\) to something else.}\fangji{Changed.}
Let $X\in\mathbb R^n$ with $\mu \Let \mathbb E[X]$ and $\Theta \Let \mathbb E[XX^\top]$. Suppose that $x_0\in\mathbb R^n$ and $t>0$. Then for any $x\in\mathbb R^n$ such that
$\|x-x_0\|<t$, the minimum value of
\begin{equation}
\frac{\mathbb E\big[\|X-x\|^2\big]}{(t-\|x-x_0\|)^2},
\label{eq:ball_reference_quantity}
\end{equation}
is attained at
\begin{align}
x^\star=
\begin{cases}
x_0+
\left(\frac{t\|o\|-c}{t-\|o\|}\right)
\frac{o}{\|o\|},
& \textrm{if } \mathcal{T} \textrm{ holds}
\\[6pt]
x_0, & \textrm{otherwise,}
\end{cases}
\label{eq:optimal_reference}
\end{align}
where
$o\Let\mu-x_0$, $c \Let \mathbb E[\|X-x_0\|^2]
= \operatorname{tr}(\Theta)-2\mu^\top x_0+\|x_0\|^2$, and $\mathcal{T}\Let \left\{\mu\neq x_0,\ \ t>\|o\|,\ \ c<t\|o\|\right\}$.
\end{lemma}

\begin{proof}
Let $y:=x-x_0$, $Z:=X-x_0$. 
Then, $\mathbb E[Z]=o$, $\mathbb E[\|Z\|^2]=c$, $\mathbb E[\|Z-y\|^2]
=c-2\mu_0^\top y+\|y\|^2$, and
\begin{align}
    \frac{\mathbb E\big[\|X-x\|^2\big]}{(t-\|x-x_0\|)^2}
    =
    \frac{\mathbb E\big[\|Z-y\|^2\big]}{(t-\|y\|)^2}=\frac{c-2c^\top y+\|y\|^2}{(t-\|y\|)^2}. \nn
\end{align}
Fix $r=\|y\|$, by the Cauchy-Schwarz Inequality \cite[Theorem 1.35]{ref:Rud-Analysis} we have $o^\T  y\le \|o\|\|y\|$, and the equality is attained when $o$ and $y$ are parallel.
Therefore, the minimum of~\ref{eq:ball_reference_quantity} is given by
\begin{align}
    \frac{c-2\|o\|r+r^2}{(t-r)^2}\eqqcolon \phi(r). \nn
\end{align}
Optimizing $r \mapsto \phi(r)$ over $0\le r <t$ yields the optimal solution 
\begin{align}
    r^\star=
    \begin{cases}
    \left(\frac{t\|o\|-c}{t-\|o\|}\right)
    \frac{o}{\|o\|}\quad &\text{if conditions in } \mathcal{T} \text{ holds},
    \\[6pt]
    0 &\text{otherwise}. \nn
    \end{cases}
\end{align}
Taking $y$ to be parallel to $o$ gives the optimal solution~\ref{eq:optimal_reference}. 
\end{proof}

To avoid ambiguity in the notation used in the proofs of Proposition~\ref{prop:adimissible_control_form} and Theorem~\ref{thm:optiaml_control_form} below, we introduce the following notation associated with Problem~\ref{problem:original}. 
Let \(u\in\mathcal U\) and let \(j\in\indexset\). 
For \(k=0,\ldots,T\), let \(x_k^u\) denote the state at time \(k\) induced by the control sequence \(u\), and let \(\mu_k^u(j)\) and \(\Sigma_k^u(j)\) denote the corresponding conditional mean and conditional covariance of \(x_k^u\) given \(q_k=j\), respectively.
We write \(J^u\) for the value of the cost functional \(J\) under the control \(u\). In addition, for \(k=0,\ldots,T-1\), define $v_k^u(j)\Let\mathbb{E}[u_k\mid q_k=j]$,
\(\Sigma_k^{uu}(j)\Let \mathbb E\!\big[(u_k-v_k^u(j))(u_k-v_k^u(j))^\T \mid q_k=j\big],
\) and \(\Sigma_k^{ux}(j)\Let \mathbb E\!\big[(u_k-v_k^u(j))(x_k^u-\mu_k^u(j))^\T \mid q_k=j\big].\)

% To avoid ambiguity of notations, in the proof of Proposition~\ref{prop:adimissible_control_form} and Theorem~\ref{thm:optiaml_control_form}, we define the below notations~\ref{problem:original}:
% Let $u\in\mathcal{U}$ and let $j\in\indexset$.
% For $k=0,\ldots,T$, denote $x_k^u$, $\mu_k^u(j)$, and $\Sigma_k^u(j)$ as the corresponding values of $x_k$, $\mu_k(j)$, and $\Sigma_k(j)$ under control $u$.
% For $k=0,\ldots,T-1$, denote $v_k^u(j)$ as the corresponding value of $v_k(j)$ under control $u$.
% We denote $J^u$ as the corresponding value of $J$ under control $u$.
% Besides, for $k=0,\ldots,T-1$, we also define $\Sigma^{uu}_k(j)\Let\mathbb{E}[(u_k-v^u_k(j))(u_k-v^u_k(j))^\T\mid q_{k}=j]$ and $\Sigma_k^{ux}(j)\Let\mathbb{E}[(u_k- v^u_k(j))(x^u_k- \mu^u_k(j))^\T \mid q_{k}=j]$.

\subsection{Proof of Proposition~\ref{prop:adimissible_control_form}}
\label{appendix:proof_of_admissible_control_form}

% \fangji{Need to clean up the notations, define $\cdot^u$ for variables.}
% \begin{proof}[Proof of Proposition~\ref{prop:adimissible_control_form}]
% \fangji{There is a more tricky proof: assume $\hat u$ steers each mode to some mean \& covariance before jump, then you can apply linear control law to the same mean \& covariance. And the mean \& covariance after jump is determined by those before jump.}

% For a variable $a$ associated with~\ref{eq:plain_mjls_system} and an admissible control $u$, we use $a^u$ to denote that it is under control $u$.
% Besides, for $k=0,\ldots,T-1$ and $j\in\indexset$, we define $v^u_k(j):=\mathbb{E}[u_k\mid q_k=j]$, $\Sigma^{uu}_k(j):=\mathbb{E}[(u_k-v^u_k(j))(u_k-v^u_k(j))^\T\mid q_{k}=j]$, and $\Sigma_k^{ux}(j)=\mathbb{E}[(u_k- v^u_k(j))(x^u_k- \mu^u_k(j))^\T \mid q_{k}=j]$.
% % Let $\hat u$ be any admissible control.

Let $u\in\mathcal{U}$ be an admissible control.
We first derive the propagation formulas of $\mu^u(\indexset)$ and $\Sigma^u(\indexset)$.
For $k=0,\ldots,T-1$ and $j\in\indexset$, we have
\begin{align}
    &\mathbb{E}[x_{k+1}^{u}\mid q_{k+1}=j]
    =\sum_{i\in\indexset}\mathbb{E}[x_{k+1}^{ u}\indicator_{\{q_k=i\}}\mid q_{k+1}=j]
    \nonumber\\
    &=\sum_{i\in\indexset}s_k^{ij}\mathbb{E}[x_{k+1}^{u}\mid q_k=i,\,q_{k+1}=j]
    =\sum_{i\in\indexset}s_k^{ij}\mathbb{E}\Big[A_k(i) x_k^{u}+B_k(i) u_k+\sum_{r=1}^{m_x} A_k^{(r)}(i) x_k^{ u}\xi_k^{(r)}
    \nonumber\\
    &\quad+ \sum_{s=1}^{m_u} B_k^{(s)}(i)  u_k \eta_k^{(s)} +D_k(i)w_k\mid q_k=i\Big]
    =\sum_{i\in\indexset}s_k^{ij}\big(A_k(i) \mu^{u}_k(i)+B_k(i) v^{u}_k(i)\big)
    \nonumber\\ &\eqqcolon F_{k,j}\left(\mu_k^u(\indexset),v_k^u(\indexset)\right),
    \nonumber
\end{align}
and
\begin{align}
    &\mathbb{E}[x_{k+1}^{ u}x_{k+1}^{ u\T}\mid q_{k+1}=j]-\mu_{k+1}^{ u}(j)\mu_{k+1}^{ u}(j)^\T
    \nonumber\\
    &=\sum_{i\in\indexset}\mathbb{E}[x_{k+1}^{ u}x_{k+1}^{ u\T}\indicator_{\{q_k=i\}}\mid q_{k+1}=j]-\mu_{k+1}^{ u}(j)\mu_{k+1}^{ u}(j)^\T
    \nonumber\\
    &=\sum_{i\in\indexset}s_k^{ij}\mathbb{E}[x_{k+1}^{ u}x_{k+1}^{ u\T}\mid q_k=i,\,q_{k+1}=j]
     - \mu_{k+1}^{ u}(j) \mu_{k+1}^{ u}(j)^\T
    \nonumber\\
    &=\sum_{i\in\indexset}s_k^{ij}\mathbb{E}\Big[\Big(A_k(i) x_k^{ u}+B_k(i) u_k+\sum_{r=1}^{m_x} A_k^{(r)}(i)x_k^{ u} \xi_k^{(r)} 
    \nonumber\\
    &\quad+ \sum_{s=1}^{m_u} B_k^{(s)}(i)  u_k \eta_k^{(s)}+D_k(i)w_k\Big)\Big(A_k(i) x_k^{ u}+B_k(i) u_k+\sum_{r=1}^{m_x} A_k^{(r)}(i) x_k^{ u}\xi_k^{(r)}
    \nonumber\\
    &\quad  + \sum_{s=1}^{m_u} B_k^{(s)}(i)  u_k \eta_k^{(s)}+D_k(i) w_k\Big)^\T
    \mid q_k=i\Big]- \mu^{ u}_{k+1}(j) \mu^{ u}_{k+1}(j)^\T 
    \nonumber\\
    &=\sum_{i\in\indexset}s_k^{ij}\Big(A_k(i) \big(\Sigma^{ u}_k(i)+\mu_k^{ u}(i)\mu_k^{ u}(i)^\T\big)A_k(i)^\T 
     +A_k(i)\big(\Sigma_k^{ ux}(i)^\T+\mu_k^{ u}(i)v_k^{ u}(i)^\T\big)B_k(i)^\T 
    \nonumber\\
    &\quad +B_k(i)\big(\Sigma_k^{ ux}(i)+v_k^{ u}(i)\mu_k^{ u}(i)^\T\big)A_k(i)^\T 
     +B_k(i)\big(\Sigma_k^{ u u}(i)+v_k^{ u}(i)v_k^{ u}(i)^\T\big) B_k(i)^\T
    \nonumber\\
    &\quad + \sum_{r=1}^{m_x} A_k^{(r)}(i) \big( \Sigma^{ u}_k(i)  +  \mu^{ u}_k(i) \mu^{ u}_k(i)^\T\big) A_k^{(r)}(i)^\T
    + \sum_{s=1}^{m_u} B_k^{(s)}(i) \big( \Sigma_k^{ u u}(i) +  v^{ u}_k(i) v^{ u}_k(i)^\T\big) B_k^{(s)}(i)^\T 
    \nonumber\\
    &\quad + D_k(i) D_k(i)^\T\Big)-\mu_{k+1}^u(j)\mu_{k+1}^u(j)^\T
    \nonumber\\
    &\eqqcolon G_{k,j}\left(\mu_k^u(\indexset), \mu_{k+1}^u(j),\Sigma_k^u(\indexset),v_k^u(\indexset), \Sigma_k^{uu}(\indexset),\Sigma_k^{ux}(\indexset) \right).
    \nonumber
\end{align}
Thus,
\begin{align}
    &\mu^u_{k+1}(j)=F_{k,j}\left(\mu_k^u(\indexset),v_k^u(\indexset)\right),\label{eq:admissible_proof_mean_propagation}
    \\
    &\Sigma_{k+1}^u(j)=
    G_{k,j}\big(\mu_k^u(\indexset), \mu_{k+1}^u(j),\Sigma_k^u(\indexset),v_k^u(\indexset), \Sigma_k^{uu}(\indexset),\Sigma_k^{ux}(\indexset) \big).
    \label{eq:admissible_proof_cov_propagation}
\end{align}
The cost function is written as
\begin{align}
    &\sum_{k=0}^{T-1}\sum_{j\in\indexset}\rho_k(j)\mathbb{E}\Big[\big( x_k^{u\T} Q_k(j) x^u_k+ u_k^\T R_k(j) u_k\big)\mid q_{k}=j\Big]
    \nonumber\\
    &=\sum_{k=0}^{T-1}\sum_{j\in\indexset} \rho_k(j)\Big(\operatorname{tr}\big(Q_k(j)\mathbb{E}[ x^u_k x_k^{u\T} \mid q_{k}=j]\big)
    +\operatorname{tr}\big(R_k(j) \mathbb{E}[ u_k u_k^\T \mid q_{k}=j]\big)\Big)
    \nonumber\\
    &=\sum_{k=0}^{T-1}\sum_{j\in\indexset} \rho_k(j)\Big(\operatorname{tr}\big(Q_k(j)( \Sigma^u_k(j)+ \mu^u_k(j) \mu^u_k(j)^\T )\big)
    +\operatorname{tr}\big(R_k(j)( \Sigma^{uu}_k(j)+ v^u_k(j) v^u_k(j)^\T )\big)\Big)
    \nonumber\\
    &\eqqcolon H\left(\mu^u(\indexset),\Sigma^u(\indexset), v^u(\indexset),\Sigma^{uu}(\indexset)\right).
    \nonumber
\end{align}
Consequently, we have
\begin{align}
J^u=H\left(\mu^u(\indexset),\Sigma^u(\indexset), v^u(\indexset),\Sigma^{uu}(\indexset)\right).
    \label{eq:hat_J}
\end{align}
Now, let $\hat u$ be a feasible control for Problem~\ref{problem:original}. 
For each $k=0,\ldots,T$ and $j\in\indexset$, let $\bar x_k\Let x_k^{\hat u}- \mu^{\hat u}_k(j)$ and $\bar u_k=\hat u_k- v^{\hat u}_k(j)$. 
Then,
\begin{align}
    &\mathbb{E}\left[\begin{pmatrix}
        \bar x_k\\
        \bar u_k
    \end{pmatrix}\begin{pmatrix}
        \bar x_k\\
        \bar u_k
    \end{pmatrix}^\T \;\bigg|\; q_k=j \right]=\begin{pmatrix}
    \Sigma^{\hat u}_k(j) & \Sigma_k^{\hat ux}(j)^\T\\
    \Sigma_k^{\hat ux}(j) & \Sigma_k^{\hat u\hat u}(j)
\end{pmatrix} \succeq 0.\nonumber
\end{align}
Applying Lemma~\ref{lemma:for_lossless_proof} under the condition $q_k=j$, there exist $K_k(j)\in\mathbb{R}^{n_u\times n_x}$ and $ V_k(j)\succeq 0$ such that 
\begin{align}
     \Sigma_k^{\hat ux}(j)&= K_k(j)\hat \Sigma^{\hat u}_k(j),
    \label{S_ux}
    \\
     \Sigma_k^{\hat u\hat u}(j)&= K_k(j) \Sigma^{\hat u}_k(j) K_k(j)^\T + V_k(j).
    \label{eq:S_uu}
\end{align}
For $k=0,\ldots,T-1$, let $\nu_k\in\R[n_u]$ be any $\controlfiltration_k$-measurable random vector such that for every $j\in\indexset$, $\nu_k$ is independent of $x_k$, conditional on $q_k=j$, such that $\mathbb{E}[\nu_k\mid q_k=j]=0$ and $\mathbb{E}[\nu_k\nu_k^\T\mid q_k=j]=V_k(j)$.
For $k=0,\ldots,T-1$, let
\begin{align}
    u_k= K_k(q_k)(x_k^u-\mu_k^u(q_k))+v_k^{\hat u}(q_k)+\nu_k.
    \label{eq:admissible_proof_u_form}
\end{align}
We will show that $u=\{u_k\}_{k=0}^{T-1}$ is a feasible control such that $J^u=J^{\hat u}$.

To this end, notice that, for $k=0,\ldots,T-1$, $u_k$ is a deterministic function of $x_k^u$, $q_k$ and $\nu_k$, which are all $\controlfiltration_k$-measurable random vectors, hence $u_k$ is also $\controlfiltration_k$-measurable.
Next, for $k=0,\ldots,T-1$ and $j\in\indexset$, from~\ref{eq:admissible_proof_u_form} we see that $\mathbb{E}[u_k\mid q_k=j]=v_k^{\hat u}(j)$, and thus
\begin{align}
    v_k^{u}(j)=v_k^{\hat u}(j).
    \label{eq:admissible_proof_v_equal}
\end{align}
Consequently, from~\ref{eq:admissible_proof_u_form} and~\ref{eq:admissible_proof_v_equal}, we may write 
% \sid{There is \(x_k\) in the calculation below. check carefully.}\fangji{I added ``from~\ref{eq:admissible_proof_u_form} and~\ref{eq:admissible_proof_v_equal}''.} \sid{sure. The \(x_k\) in line 3 should be \(x_k^u\) right?}\fangji{Yes. Changed.} \sid{ok}
\begin{align}
    &\mathbb{E}\left[(u_k-v_k^u(j))(u_k-v_k^u(j))^\T\mid q_k=j\right]
    =\mathbb{E}\left[(u_k-v_k^{\hat u}(j))(u_k-v_k^{\hat u}(j))^\T\mid q_k=j\right]
    \nonumber\\
    &=\mathbb{E}\big[\big(K_k(j)(x_k^u-\mu_k^u(j))+\nu_k\big)\big(K_k(j)(x_k^u-\mu_k^u(j))
    +\nu_k\big)^\T\mid q_k=j\big]\nn \\& =K_k(j)\Sigma_k^u(j)K_k(j)^\T+V_k(j),
    \nonumber
\end{align}
and
\begin{align}
    &\mathbb{E}\left[(u_k-v_k^u(j))(x_k^u-\mu_k^u(j))^\T\mid q_k=j\right]
    =\mathbb{E}\left[(u_k-v_k^{\hat u}(j))(x_k^u-\mu_k^u(j))^\T\mid q_k=j\right]
    \nonumber\\
    &=\mathbb{E}\left[\left(K_k(j)(x_k^u-\mu_k^u(j))+\nu_k\right)(x_k^u-\mu_k^u(j))^\T\right]
    =K_k(j)\Sigma_k^{u}(j).
    \nonumber
\end{align}
It follows that
\begin{align}
    &\Sigma_k^{uu}(j)=K_k(j)\Sigma_k^u(j)K_k(j)^\T+V_k(j),
    \label{eq:admissible_proof_S_uu}
    \\
    &\Sigma_k^{ux}(j)=K_k(j)\Sigma_k^{u}(j).
\end{align}

By construction, the state process induced by $u$ is initialized with the same initial random state as the state process induced by $\hat u$. 
Hence, \(x_0^{u}=x_0^{\hat u}\) and, therefore, the initial condition \eqref{eq:initial_condition} holds for $x_0^{u}$ as well.
% Assume that $x_0^u$ satisfies the initial condition~\ref{eq:initial_condition}.
From~\ref{eq:admissible_proof_mean_propagation} and~\ref{eq:admissible_proof_v_equal}, we know that for each $k=0,\ldots,T$ and $j\in\indexset$,
\begin{align}
    \mu_k^u(j)=\mu_k^{\hat u}(j).
    \label{eq:admissible_proof_mu_equal}
\end{align}
Furthermore, from~\ref{eq:admissible_proof_cov_propagation},~\ref{S_ux}--\ref{eq:S_uu} and~\ref{eq:admissible_proof_v_equal}--\ref{eq:admissible_proof_mu_equal}, it can be easily proved by induction that, for each $k=0,\ldots,T-1$ and $j\in\indexset$, we have
\begin{align}
\hspace*{-2mm}
    \Sigma_k^{u}(j)=\Sigma_k^{\hat u}(j),~
    \Sigma_k^{uu}(j)=\Sigma_k^{\hat u\hat u}(j),~
    \Sigma_k^{ux}(j)=\Sigma_k^{\hat ux}(j),
    \label{eq:admissible_form_proof_same_trajectories}
\end{align}
and $\Sigma_T^{u}(j)=\Sigma_T^{\hat u}(j)$.
Therefore, $x_T^u$ also satisfies the terminal condition~\ref{eq:terminal_condition}, implying that $u$ is feasible. Furthermore, from~\ref{eq:hat_J} along with~\ref{eq:admissible_proof_mu_equal}--\ref{eq:admissible_form_proof_same_trajectories}, we obtain $J^u=J^{\hat u}$, and the proof is complete. \qed

\subsection{Proof of Theorem~\ref{thm:optiaml_control_form}}\label{appendix:proof_of_optimal_control_form}
Let $u\in\mathcal{U}$ be an admissible control.
For $k=1,\ldots,T$ and $j\in\indexset$, define $m_k^{u}(j)\Let \mathbb{E}[x_k^u\mid q_{k-1}=j]$ and $S_k^{u}(j)\Let \mathbb{E}[(x_k^u-m_k^u(j))(x_k^u-m_k^u(j))^\T\mid q_{k-1}=j]$.
Then, for $k=0,\ldots,T-1$ and $j\in\indexset$,
\begin{align}
    \mu_{k+1}^{u}(j)&=\sum_{i\in\indexset}\mathbb{E}[x^{u}_{k+1}\indicator_{\{q_k=i\}}\mid q_{k+1}=j]
    =\sum_{i\in\indexset}s_k^{ij}\mathbb{E}[x^{u}_{k+1}\mid q_k=i,\,q_{k+1}=j]
    \nonumber\\
    &=\sum_{i\in\indexset}s_k^{ij}m_{k+1}^{u}(i),
    % =\sum_{i\in\indexset}s_k^{ij}m_{k+1}^{\hat u}(i),
    \label{eq:optimal_form_proof_mean_k+1_ustar}
\end{align}
and
\begin{align}
    &\Sigma_{k+1}^{u}(j)=\sum_{i\in\indexset}\mathbb{E}[x_{k+1}^{u}x_{k+1}^{u\T}\indicator_{\{q_k=i\}}\mid q_{k+1}=j]
     -\mu_{k+1}^{u}(j) \mu_{k+1}^{u}(j)^\T
    \nonumber\\
    &=\sum_{i\in\indexset}s_k^{ij}\mathbb{E}[x_{k+1}^{u}x_{k+1}^{u\T}\mid q_k=i,\, q_{k+1}=j]
    -\mu_{k+1}^{u}(j) \mu_{k+1}^{u}(j)^\T
    \nonumber\\
    &=\sum_{i\in\indexset}s_k^{ij}\mathbb{E}[x_{k+1}^{u}x_{k+1}^{u\T}\mid q_k=i]-\mu_{k+1}^{u}(j)\mu_{k+1}^{u}(j)^\T
    \nonumber\\
    &=\sum_{i\in\indexset}s_k^{ij}\mathbb{E}[(x_{k+1}^{u}-m_{k+1}^{u}(i))(x_{k+1}^{u}-m_{k+1}^{u}(i))^\T\mid q_k=i]
    \nonumber\\
    &\quad +\sum_{i\in\indexset}s_k^{ij}m_{k+1}^{u}(i)m_{k+1}^{u}(i)^\T-\mu_{k+1}^{u}(j)\mu_{k+1}^{u}(j)^\T
    \nonumber\\
    &=\sum_{i\in\indexset}s_k^{ij}S_{k+1}^{u}(i)+\sum_{i\in\indexset}s_k^{ij}m_{k+1}^{u}(i)m_{k+1}^{u}(i)^\T -\mu_{k+1}^{u}(j) \mu_{k+1}^{u}(j)^\T,
    \label{eq:optimal_form_proof_cov_k+1_ustar}
\end{align}

From Assumption~\ref{assumption:feasible_and_sigma_positive}, there exists an optimal control $\hat u$ for Problem~\ref{problem:original}.

Let now $x_0^{u^\star}$ satisfy the initial condition~\ref{eq:initial_condition}.
We construct $u^\star=\{u^\star_k\}_{k=0}^{T-1}$ by induction on $k$ such that for each $k=0,\ldots,T$, the following holds
\begin{align}
    \mu_k^{u^\star}(j)=\mu^{\hat u}_k(j),\quad \Sigma_k^{u^\star}(j)=\Sigma_k^{\hat u}(j),\quad j\in\indexset.
    \label{eq:optimal_form_proof_same_trajectories}
\end{align}
For $k=0$,~\ref{eq:optimal_form_proof_same_trajectories} follows directly from~\ref{eq:initial_condition}.
For time $k\in\{0,\ldots,T-1\}$, assume that $u_0^\star,\ldots,u_{k-1}^\star$ have been constructed such that~\ref{eq:optimal_form_proof_same_trajectories} holds.
% \red{
Then, for each $j\in\indexset$, let $\Omega'=\{q_k=j\}$, $\statefiltration\Let \aset[]{A\cap\Omega' \suchthat A\in\statefiltration}$, $\statefiltration_0'\Let \aset[]{A\cap\Omega' \suchthat A\in\statefiltration_k}$, $\statefiltration_1'\Let \aset[]{A\cap\Omega' \suchthat A\in\statefiltration_{k+1}}$, and $\mathbb{P}(\cdot)\Let\mathbb{P}(\,\cdot\,|q_k=j)$.
Consider the single-mode one-step covariance steering subproblem at time $k$, given by
\begin{align}
    \min_{\gamma\in \Gamma}\quad &J^\gamma_{k,j}\Let \mathbb{E}[y_0^{\T} Q_k(j)y_{0}+\gamma^\T R_k(j)\gamma],
    \label{eq:one_step_problem_begin}
    \\
    \text{s.t.} \quad&y_1=A_k(j)y_{0}+B_k(j)\gamma+D_k(j)w_k|_{\Omega'},
    \label{eq:one_step_transition}
    \\
    &\mu_\mathrm{in}'=\mu_k^{\hat u}(j),\quad \Sigma_\mathrm{in}'=\Sigma_k^{\hat u}(j),
    \label{eq:one_step_subproblem_initial_condition}
    \\
    &\mu_\mathrm{out}'=m_{k+1}^{\hat u}(j), \quad \Sigma_\mathrm{out}'=S_{k+1}^{\hat u}(j),
    \label{eq:one_step_problem_end}
\end{align}
where, to avoid ambiguity in the notation, we use $y$ to denote the state and $\gamma$ to denote the control for this subproblem.
We assume that $\{y_\tau\}_{\tau=0}^1$ is $\{\statefiltration'_\tau\}_{\tau=0}^1$-adapted, and $\Gamma$ is the set of all $\R[n_u]$-valued, $\statefiltration_0'$-measurable, square-integrable random vectors.
One can also verify that $w_k|_{\Omega'}$ is a zero-mean, identity-covariance random vector that is independent of $\statefiltration_0'$.
Therefore, the subproblem~\eqref{eq:one_step_problem_begin}--\eqref{eq:one_step_problem_end} becomes a standard covariance steering problem on the filtered probability space $(\Omega', \statefiltration',\{\statefiltration_\tau'\}_{\tau=0}^1,\mathbb{P}')$.
From~\cite[Theorem 1]{liu2024reachability}, the feasibility of the subproblem depends only on the boundary conditions~\ref{eq:one_step_subproblem_initial_condition} and~\ref{eq:one_step_problem_end} instead of the exact initial distribution.
For the specific \(\mathcal{F}_0'\)-measurable initial state \(y_0 = x_k^{\hat u}\big|_{\Omega'}\), the \(\statefiltration_0'\)-measurable control \(\gamma = \hat u_k\big|_{\Omega'}\) generates, through the one-step dynamics \eqref{eq:one_step_transition}, the terminal state \(y_1 = x_{k+1}^{\hat u}\big|_{\Omega'}\), and satisfies the boundary conditions \eqref{eq:one_step_subproblem_initial_condition}--\eqref{eq:one_step_problem_end}.

% \blue{Fangji, please check this: For the specific \(\mathcal{F}_0'\)-measurable initial state \(y_0 = x_k^{\hat u}\big|_{\Omega'}\), the \(\statefiltration_0'\)-measurable control \(\gamma = \hat u_k\big|_{\Omega'}\) generates, through the one-step dynamics \eqref{eq:one_step_transition}, the terminal state \(y_1 = x_{k+1}^{\hat u}\big|_{\Omega'}\), and satisfies the boundary conditions \eqref{eq:one_step_subproblem_initial_condition}--\eqref{eq:one_step_problem_end}.}

Therefore, the subproblem~\eqref{eq:one_step_problem_begin}--\eqref{eq:one_step_problem_end} is feasible.
Applying~\cite[Theorem 3]{liu2024optimal} to subproblem~\eqref{eq:one_step_problem_begin}--\eqref{eq:one_step_problem_end}, there exist $K_k(j)\in\R[n_u\times n_x]$ and $v_k(j)\in\R[n_u]$ such that 
\begin{align}
    \gamma^\star=K_{k}(j)\big(y_0-\mu'_\mathrm{in}\big)+v_k(j),
    \nonumber
\end{align}
% \blue{the \(y_k\) is for general kth-problem or (63)-(66)? becuuse in (63)-(66) there is not \(y_k\), it has states \(y_0,y_1\). Thus \(\gamma^{\star}\) is a little confusing}\fangji{Oh, it should be $y_0$. I forgot to change this.}\sid{okay got it}
is an optimal control for this subproblem.
Let now
\begin{align}
    u_k^\star=K_k(q_k)\big(x_k^{u^\star}-\mu_k^{u^\star}(q_k)\big)+v_k^{u^\star}(q_k).
    \nonumber
\end{align}
% \blue{note that \(v_k^{u^\star}(q_k)\) was not define in the appendix, or it was? I can't find it.}\fangji{It is defined before proposition 1} \sid{ah okay}
Then, under the control $u^\star$, $x_{k+1}^{u^\star}|_{\Omega'}$  satisfies the terminal condition~\ref{eq:one_step_problem_end}, that is,
% \blue{\(x_{k+1}^{u^\star}|_{\Omega'}\) is restriction to \(\Omega'\)?}\fangji{Yes.}\sid{got it}
\begin{align}
    m_{k+1}^{u^\star}(j)=m_{k+1}^{\hat u}(j),\quad S_{k+1}^{u^\star}(j)=S_{k+1}^{\hat u}(j).
    \nonumber
\end{align}
Hence, from~\ref{eq:optimal_form_proof_mean_k+1_ustar} and~\ref{eq:optimal_form_proof_cov_k+1_ustar}, we know that~\ref{eq:optimal_form_proof_same_trajectories} holds for $k+1$.

From~\ref{eq:optimal_form_proof_same_trajectories}, we know that $u^\star$ is feasible to Problem~\ref{problem:original}.
Since for each $k=0,\ldots,T-1$, $u_k^\star$ is constructed optimally from the subproblems, it is straightforward to see that $J^{u^\star}\le J^{\hat u}$.
Hence, $u^\star$ is an optimal control for Problem~\ref{problem:original}.

\subsection{Proof of Proposition~\ref{proposition:ball_constraints}}
\label{appendix:proof_ball_constraints}

For the constraints on $x_\constridx$, let $a_\constridx^\mathrm{ref}(j)\in\mathbb{R}^{n_x}$ be such that $r_{\constridx,x}> \|a_\constridx^\mathrm{ref}(j)-a_\constridx\|$, for $j\in\indexset$. 
Then,
\begin{align}
    &\mathbb{P}(\|x_\constridx-a_\constridx\|>r_{\constridx,x})
    =\sum_{j\in\indexset}\rho_\constridx(j)\mathbb{P}(\|x_\constridx-a_\constridx\|>r_{\constridx,x}\mid q_\constridx=j)
    \nonumber\\
    &\le \sum_{j\in\indexset}\rho_\constridx(j)\mathbb{P}(\|x_\constridx-a_\constridx^\mathrm{ref}(j)\|>r_{\constridx,x}-\|a_\constridx^\mathrm{ref}(j)-a_\constridx\| \mid q_\constridx=j)
    \nonumber\\
    &= \sum_{j\in\indexset}\rho_\constridx(j)\mathbb{P}(\|x_\constridx-a_\constridx^\mathrm{ref}(j)\|^2>(r_{\constridx,x}-\|a_\constridx^\mathrm{ref}(j)-a_\constridx\|)^2 \mid q_\constridx=j).
    \label{eq:ball_constraints_proof_temp1}
\end{align}
By Markov's Inequality~\cite{grimmett2020probability},
\begin{align}
    &\mathbb{P}\left(\|x_\constridx-a_\constridx^\mathrm{ref}(j)\|^2>(r_{\constridx,x}-\|a_\constridx^\mathrm{ref}(j)-a_\constridx\|)^2\mid q_\constridx=j\right)
    \nonumber\\
    &\le \frac{\mathbb{E}[\|x_\constridx-a_\constridx^\mathrm{ref}(j)\|^2\mid q_\constridx=j]}{(r_{\constridx,x}-\|a_\constridx^\mathrm{ref}(j)-a_\constridx\|)^2}=\frac{\operatorname{tr}\big(I_\constridx(j)\big)-2a_\constridx^\mathrm{ref}(j)^\T \mu_\constridx(j)+a_\constridx^\mathrm{ref}(j)^\T a_\constridx^\mathrm{ref}(j))}{(r_{\constridx,x}-\|a_\constridx^\mathrm{ref}(j)-a_\constridx\|)^2}.
    \label{eq:ball_proof_1}
\end{align}
From~\ref{eq:ball_constraints_proof_temp1} and~\ref{eq:ball_proof_1}, it is straightforward to show that~\ref{eq:ball_constraints_on_x} is implied by~\ref{eq:convexified_ball_constraints_for_x}.

It can be similarly shown that for $b_k^\mathrm{ref}(j)\in\mathbb{R}^{n_u}$ such that $r_{k,u} > \|b_k^\mathrm{ref}(j)-b_k\|$,~\ref{eq:ball_constraints_on_u} is implied by~\ref{eq:convexified_ball_constraints_for_u}.

\subsection{Proof of Proposition~\ref{proposition:polytope_constraints}}
\label{appendix:proof_polytope_constraints}

To show~\ref{eq:polytope_constraints_on_x}, it suffices to show that
\begin{equation}
    \mathbb{P}(\alpha_{\constridx,x}^{\T}x_\constridx>\beta_{\constridx,x}\mid q_\constridx=j)\le \varepsilon_{x},\quad j\in\indexset.
    \label{eq:polytope_proof_temp1}
\end{equation}
For $j\in\indexset$, 
$\mathbb{E}[\alpha_{\constridx,x}^\T x_\constridx]=\alpha_{\constridx,x}^\T\mu_\constridx(j)$,
and $\operatorname{Var}(\alpha_{\constridx,x}^\T x_\constridx\mid q_\constridx=j)=\alpha_{\constridx,x}^\T(\mathbb{E}[x_\constridx x_\constridx^\T\mid q_\constridx=j]-\mu_\constridx(j)\mu_\constridx^\T(j))\alpha_{\constridx,x}=\alpha_{\constridx,x}^\T I_\constridx(j)\alpha_{\constridx,x}-(\alpha_{\constridx,x}^\T\mu_\constridx(j))^2$.
By Cantelli's Inequality~\cite{grimmett2020probability}, $\beta_{\constridx,x} \ge \alpha_{\constridx,x}^{\T}\mu_\constridx(j)$, implies that
\begin{align}
    \mathbb{P}(\alpha_{\constridx,x}^{\T}x_\constridx>\beta_{\constridx,x}\mid q_\constridx=j)&\le \frac{\mathrm{Var}(\alpha_{\constridx,x}^{\T}x_\constridx\mid q_\constridx=j)}{\mathrm{Var}(\alpha_{\constridx,x}^{\T}x_\constridx\mid q_\constridx=j)+(\beta_{\constridx,x}-\alpha_{\constridx,x}^\T\mu_\constridx(j))^2}
    \nonumber\\
    &=\frac{\alpha_{\constridx,x}^\T I_\constridx(j)\alpha_{\constridx,x}-(\alpha_{\constridx,x}^\T\mu_\constridx(j))^2}{\alpha_{\constridx,x}^\T I_\constridx(j)\alpha_{\constridx,x}-2\beta_{\constridx,x}\alpha_{\constridx,x}^\T\mu_\constridx(j)+\beta_{\constridx,x}^2}.
    \label{eq:polytope_proof_temp3}
\end{align}
Note that for $\mu_\constridx^\mathrm{ref}(j) \in\mathbb{R}^{n_x}$,
% \begin{align}
%     (\mu_\constridx(j)-\mu_\constridx^\mathrm{ref}(j) )^\T \alpha_{\constridx,x} \alpha_{\constridx,x}^\T(\mu_\constridx(j)-\mu_\constridx^\mathrm{ref}(j) )\ge 0.    
% \end{align}
% Hence
\begin{align}
    (\alpha_{\constridx,x}^\T\mu_\constridx(j))^2\ge 2\mu_\constridx^\mathrm{ref}(j)^\T \alpha_{\constridx,x} \alpha_{\constridx,x}^\T\mu_\constridx(j)-(\alpha_{\constridx,x}^\T\mu_\constridx^\mathrm{ref}(j))^2.
    \label{eq:polytope_proof_temp2}
\end{align}
Plugging~\ref{eq:polytope_proof_temp2} into~\ref{eq:polytope_proof_temp3}, yields that $\mathbb{P}(\alpha_{\constridx,x}^{\T}x_\constridx>\beta_{\constridx,x}\mid q_\constridx=j)$ is upper bounded by
\begin{align}
    \frac{\alpha_{\constridx,x}^\T I_\constridx(j)\alpha_{\constridx,x}-2\mu_\constridx^{\mathrm{ref}}(j)^\T \alpha_{\constridx,x} \alpha_{\constridx,x}^\T\mu_\constridx(j)+(\alpha_{\constridx,x}^\T\mu_\constridx^\mathrm{ref}(j))^2}{\alpha_{\constridx,x}^\T I_\constridx(j)\alpha_{\constridx,x}-2\beta_{\constridx,x}\alpha_{\constridx,x}^\T\mu_\constridx(j)+\beta_{\constridx,x}^2}.
    \label{eq:polytope_proof_temp4}
\end{align}
Therefore, to show~\ref{eq:polytope_proof_temp1}, it suffices to show that the quantity in~\ref{eq:polytope_proof_temp4} is less than or equal to $\varepsilon_x$, which this is equivalent to~\ref{eq:convexified_polytope_constraints_on_x_1}.

The proof for~\ref{eq:convexified_polytope_constraints_on_u_1}-\ref{eq:convexified_polytope_constraints_on_u_2} implying~\ref{eq:polytope_constraints_on_u} is similar and is thus omitted.\qed

\subsection{Expressions for $E_{\constridx,i}(\indexset)$, $F_{\constridx,i}(\indexset)$, $G_{\constridx,i}(\indexset)$, $c_{\constridx,i}$ and $d_{\constridx,i}$.}\label{appen:expressions}

From~\ref{eq:I_with_I_tilde}, we know that
\begin{align}
    &I_\constridx(j)=S_1\tilde I_\constridx(j)S_1^\T,\label{eq:proof_I_Itilde}
    \\
    &\mu_\constridx(j)=S_1\tilde I_\constridx(j)S_2,\label{eq:proof_mu_Itilde}
    \\
    &v_\constridx(j)=\tilde U_\constridx(j)S_2,\label{eq:proof_v_Utilde}
    \\
    &\operatorname{tr}\big(I_\constridx(j)\big)=\operatorname{tr}\big(\tilde I_\constridx(j)\big)-1,\label{eq:proof_I_Itilde_trace}
\end{align}
where
$S_1=\begin{pmatrix}
    \identityMatrix_{n_x} & 0_{n_x\times 1}
\end{pmatrix}$, and $S_2=\begin{pmatrix}
    0_{n_x\times 1}\\
    1
\end{pmatrix}$.

\subsubsection{Ball constraints on x}

Using~\ref{eq:proof_I_Itilde_trace} and~\ref{eq:proof_mu_Itilde}, the left-hand side of~\ref{eq:convexified_ball_constraints_for_x} equals to
\begin{align}
    &\sum_{j\in\indexset}\frac{\rho_\constridx(j)}{(r_{\constridx,x}-\|a_\constridx^\mathrm{ref}(j)-a_\constridx\|)^2}\Big(\operatorname{tr}\big(\tilde I_\constridx(j)\big)
    -2a_\constridx^\mathrm{ref}(j)^\T S_1 \tilde I_\constridx(j)S_2 +a_\constridx^\mathrm{ref}(j)^\T a_\constridx^\mathrm{ref}(j))-1\Big).
    \nonumber
\end{align}
Using the property of trace, this can be written as
\begin{align}
    &\sum_{j\in\indexset}\frac{\rho_\constridx(j)}{(r_{\constridx,x}-\|a_\constridx^\mathrm{ref}(j)-a_\constridx\|)^2}\Big(\operatorname{tr}\big(\tilde I_\constridx(j)\big) 
    -2\operatorname{tr}\big(S_2a_\constridx^\mathrm{ref}(j)^\T  S_1\tilde I_\constridx(j)\big) +a_\constridx^\mathrm{ref}(j)^\T a_\constridx^\mathrm{ref}(j))-1\Big)
    \nonumber\\
    &=\sum_{j\in\indexset}\frac{\rho_\constridx(j)}{(r_{\constridx,x}-\|a_\constridx^\mathrm{ref}(j)-a_\constridx\|)^2}\Big(\operatorname{tr}\Big(\big(\identityMatrix_{n_x+1}-2S_2a_\constridx^\mathrm{ref}(j)^\T S_1\big) \tilde I_\constridx(j)\Big)+a_\constridx^\mathrm{ref}(j)^\T a_\constridx^\mathrm{ref}(j))-1\Big).
    \nonumber
\end{align}
This corresponds to $\ell_s=1$ in~\ref{eq:final_constraints_on_I},
and for $j\in\indexset$
\begin{align}
    &E_{\constridx,1}(j)=\frac{\rho_\constridx(j)}{(r_{\constridx,x}-\|a_\constridx^\mathrm{ref}(j)-a_\constridx\|)^2}\big(\identityMatrix_{n_x+1}-2S_2a_\constridx^\mathrm{ref}(j)^\T S_1\big),
    \nonumber\\    &c_{\constridx,1}=\varepsilon_x-\sum_{j\in\indexset}\frac{\rho_\constridx(j)(a_\constridx^\mathrm{ref}(j)^\T a_\constridx^\mathrm{ref}(j)-1)}{(r_{\constridx,x}-\|a_\constridx^\mathrm{ref}(j)-a_\constridx\|)^2}.
    \nonumber
\end{align}

\subsubsection{Ball constraints on u}

From~\ref{eq:proof_v_Utilde}, the left-hand side of~\ref{eq:convexified_ball_constraints_for_u} equals 
\begin{align}
    &\sum_{j\in\indexset}\frac{\rho_\constridx(j)}{(r_{\constridx,u}-\|b_\constridx^\mathrm{ref}(j)-b_\constridx\|)^2}\Big(\operatorname{tr}(Y_{\constridx}(j))
      -2b_\constridx^\mathrm{ref}(j)^\T \tilde U_\constridx(j)S_2+b_\constridx^\mathrm{ref}(j)^\T b_\constridx^\mathrm{ref}(j)\Big),
     \nonumber
\end{align}
which then equals
\begin{align}
    &\sum_{j\in\indexset}\frac{\rho_\constridx(j)}{(r_{\constridx,u}-\|b_\constridx^\mathrm{ref}(j)-b_\constridx\|)^2}\Big(\operatorname{tr}(Y_{\constridx}(j))
      -2\operatorname{tr}\big(S_2b_\constridx^\mathrm{ref}(j)^\T  \tilde U_\constridx(j)\big)+b_\constridx^\mathrm{ref}(j)^\T b_\constridx^\mathrm{ref}(j)\Big).
     \nonumber
\end{align}
This corresponds to $t_s=1$ in~\ref{eq:final_constraints_on_U_Y},
and for $j\in\indexset$,
\begin{align}
    &F_{\constridx,1}(j)=\frac{\rho_\constridx(j)}{(r_{\constridx,u}-\|b_\constridx^\mathrm{ref}(j)-b_\constridx\|)^2}\, \identityMatrix_{n_u},
    \nonumber\\
    &G_{\constridx,1}(j)=-\frac{2\rho_\constridx(j)}{(r_{\constridx,u}-\|b_\constridx^\mathrm{ref}(j)-b_\constridx\|)^2}S_2b_\constridx^\mathrm{ref}(j)^\T,
    \nonumber\\
    &d_{\constridx,1}=\varepsilon_u-\sum_{j\in\indexset}\frac{\rho_\constridx(j)b_\constridx^\mathrm{ref}(j)^\T b_\constridx^\mathrm{ref}(j)}{(r_{\constridx,u}-\|b_\constridx^\mathrm{ref}(j)-b_\constridx\|)^2}.
    \nonumber
\end{align}

\subsubsection{Halfspace constraints on x}
From~\ref{eq:proof_I_Itilde} and~\ref{eq:proof_mu_Itilde}, the left-hand side of~\ref{eq:convexified_polytope_constraints_on_x_1} equals 
\begin{align}
    &(1-\varepsilon_x)\alpha_{\constridx,x}^\T S_1\tilde I_\constridx(j) S_1^\T\alpha_{\constridx,x}-2(\mu_\constridx^\mathrm{ref}(j)^\T \alpha_{\constridx,x}-\varepsilon_x\beta_{\constridx,x})\alpha_{\constridx,x}^\T 
     \times S_1\tilde I_\constridx(j)S_2.
    \nonumber
\end{align}
Using the properties of  the trace, this equals
\begin{align}
    &(1-\varepsilon_x)\operatorname{tr}\big(S_1^\T\alpha_{\constridx,x}\alpha_{\constridx,x}^\T S_1\tilde I_\constridx(j)\big)-2(\mu_\constridx^\mathrm{ref}(j)^\T \alpha_{\constridx,x}-\varepsilon_x\beta_{\constridx,x})
    \times\operatorname{tr}\big(S_2\alpha_{\constridx,x}^\T S_1\tilde I_\constridx(j)\big)
    \nonumber\\
    &=\operatorname{tr}\Big(\big((1-\varepsilon_x)S_1^\T\alpha_{\constridx,x}\alpha_{\constridx,x}^\T S_1+2(\varepsilon_x\beta_{\constridx,x}-\mu_\constridx^\mathrm{ref}(j)^\T \alpha_{\constridx,x})
    \times S_2\alpha_{\constridx,x}^\T S_1\big)\tilde I_\constridx(j)\Big).
    \nonumber
\end{align}
From~\ref{eq:proof_mu_Itilde}, the left-hand side of~\ref{eq:convexified_polytope_constraints_on_x_2} equals 
\begin{align}
    \alpha_{\constridx,x}^\T S_1\tilde I_\constridx(j)S_2=\operatorname{tr}\big(S_2\alpha_{\constridx,x}^\T S_1\tilde I_\constridx(j)\big).
    \nonumber
\end{align}
This corresponds to  $\ell_\constridx=2N_\indexset$ in~\ref{eq:final_constraints_on_I}, and for $i=1,\ldots,N_\indexset$,
\begin{align}
    &E_{\constridx,2i-1}(j)=
    \begin{cases}
        (1-\varepsilon_x)S_1^\T\alpha_{\constridx,x}\alpha_{\constridx,x}^\T S_1+2(\varepsilon_x\beta_{\constridx,x}
        \\
        \quad-\mu_\constridx^\mathrm{ref}(j)^\T  \alpha_{\constridx,x})S_2\alpha_{\constridx,x}^\T S_1 \quad \text{for } j=i,
        \\[6pt]
        0,\quad \text{otherwise},
    \end{cases}
    \nonumber\\
    &E_{\constridx,2i}(j)\begin{cases}
        S_2\alpha_{\constridx,x}^\T S_1, & j=i
        \\
        0,& \text{otherwise},
    \end{cases}
    \nonumber\\
    &c_{\constridx,2i-1}=\varepsilon_{x}\beta_{\constridx,x}^2-(\alpha_{\constridx,x}^\T\mu_\constridx^\mathrm{ref}(i))^2,\quad c_{\constridx,2i}=\beta_{\constridx,x}.
    \nonumber
\end{align}

\subsubsection{Halfspace constraints on u}
From~\ref{eq:proof_v_Utilde}, the left-hand side of~\ref{eq:convexified_polytope_constraints_on_u_1} equals to
\begin{align}
    &(1-\varepsilon_u)\alpha_{\constridx,u}^\T Y_{\constridx}(j)\alpha_{\constridx,u} 
     - 2\big(v_\constridx^\mathrm{ref}(j)^\T\alpha_{\constridx,u}-\varepsilon_u\beta_{\constridx,u}\big)\alpha_{\constridx,u}^\T \tilde U_\constridx(j)S_2,
    \nonumber
\end{align}
which can be further written as
\begin{align}
    &(1-\varepsilon_u)\operatorname{tr}\big(\alpha_{\constridx,u}\alpha_{\constridx,u}^\T Y_{\constridx}(j)\big) 
    - 2\big(v_\constridx^\mathrm{ref}(j)^\T\alpha_{\constridx,u}-\varepsilon_u\beta_{\constridx,u}\big) \operatorname{tr}\big(S_2\alpha_{\constridx,u}^\T\tilde U_\constridx(j)\big).
    \nonumber
\end{align}
From~\ref{eq:proof_v_Utilde}, the left-hand side of~\ref{eq:convexified_polytope_constraints_on_u_2} equals to
\begin{align}
    \alpha_{\constridx,u}^\T\tilde U_\constridx(j)S_2=\operatorname{tr}\big(S_2\alpha_{\constridx,u}^\T\tilde U_\constridx(j)\big).
    \nonumber
\end{align}
This corresponds to  $t_\constridx=2N_\indexset$ in~\ref{eq:final_constraints_on_U_Y}, and for $i=1,\ldots,N_\indexset$ 
\begin{align}
    &F_{\constridx,2i-1}(j)=\begin{cases}
        (1-\varepsilon_u)\alpha_{\constridx,u}\alpha_{\constridx,u}^\T, & j=i,\\
        0, &\text{otherwise},
    \end{cases}
    \nonumber\\
    &G_{\constridx,2i-1}(j)=\begin{cases}
         2\big(\varepsilon_u\beta_{\constridx,u}-v_\constridx^\mathrm{ref}(j)^\T\alpha_{\constridx,u}\big)S_2\alpha_{\constridx,u}^\T, \quad  j=i,
         \\
         0, \quad \text{otherwise},
    \end{cases}
    \nonumber\\
    &d_{\constridx,2i-1}=\varepsilon_{u}\beta_{\constridx,u}^2-(\alpha_{\constridx,u}^\T v_\constridx^\mathrm{ref}(i) )^2,\quad F_{\constridx,2i}(j)=0,\quad j\in\indexset,
    \nonumber\\
    &G_{\constridx,2i}(j)=\begin{cases}
        S_2\alpha_{\constridx,u}^\T, & j=i,
        \\
        0,&\text{otherwise},
    \end{cases}
    \nonumber\\
    &d_{\constridx,2i}=\beta_{\constridx,u}.
    \nonumber
\end{align}

% \subsection{Optimal reference for ball constraints}
% \label{appendix:proof_ball_reference}